\documentclass[10pt, a4paper]{amsart}
\usepackage{article}

\title[Codimension one partially hyperbolic diffeomorphisms]{Codimension one partially hyperbolic diffeomorphisms with a uniformly compact center foliation}
\author[Bohnet]{Doris Bohnet}
\address{Universit\'e de Bourgogne, France.}
\email{Doris.Bohnet@u-bourgogne.fr}
\thanks{This paper was partially supported by the \emph{Forschungsfond} of the Department of Mathematics and the \emph{K\"orperschaftsverm\"ogen} from the Universit\"at Hamburg as well as by the mathematical institut IMB of the Universit\'e de Bourgogne and substantially by the \emph{Conseil R\'egional de Bourgogne} thanks to a postdoctoral scholarship.}
\date{\today}

\begin{document}

\begin{abstract}
We consider a partially hyperbolic $C^1$-diffeo\-morphism $f: M \rightarrow M$ with a uniformly compact $f$-invariant center foliation $\mathcal{F}^c$. We show that if the unstable bundle is one-dimensional and oriented, then the holonomy of the center foliation vanishes everywhere, the quotient space $M/\mathcal{F}^c$ of the center foliation is a torus and $f$ induces a hyperbolic automorphism on it, in particular, $f$ is centrally transitive.\\
We actually obtain further interesting results without restrictions on the unstable, stable and center dimension: we prove a kind of spectral decomposition for the chain recurrent set of the quotient dynamics, and we establish the existence of a holonomy invariant family of measures on the unstable leaves (Margulis measure).   
\end{abstract}
\maketitle

\section{Introduction}

The aim of this article is a classification of partially hyperbolic systems with one-dimensional unstable (or stable) direction under the additional assumption of a uniformly compact center foliation. A $C^1$-diffeomorphism $f: M \rightarrow M$  is called \textit{partially hyperbolic} if there exists an invariant, non-trivial decomposition of the tangent bundle $$TM = E^s \oplus E^{c} \oplus E^{u},\quad d_xfE^{\alpha}(x)=E^{\alpha}\left(f(x)\right), \;\alpha=s,c,u,$$ and if for every $x \in M$ and all unit vectors $v^s \in E^s(x), v^u \in E^u(x)$ and $v^c \in E^c(x)$ it holds
$$\left\|d_xf(v^s) \right\| < \left\|d_xf(v^c)\right\| < \left\|d_xf(v^u)\right\|, \; \left\|d_xf(v^s) \right\| < 1 < \left\|d_xf(v^u)\right\|.$$ 
The stable and unstable bundles $E^s$ and $E^u$ are uniquely integrable to $f$-invariant stable and unstable foliations $\mathcal{F}^s$ and $\mathcal{F}^u$. By contrast, the center bundle $E^c$ is in general not uniquely or even weakly integrable (cf. \cite{W98,HHU10a}). We assume in the following that there exists a uniformly compact $f$-invariant center foliation $\mathcal{F}^c$ everywhere tangent to the center bundle $E^c$ where a foliation is called \emph{uniformly compact} if every leaf is a compact manifold and the leaf volume is bounded (see Section~\ref{s.background}).\\
Recall that every Anosov diffeomorphism with a one-dimensional unstable direction is conjugate to a hyperbolic toral automorphism (see \cite{F70,N70}). The hypo\-thesis of a compact center foliation allows us to recover (topological) hyperbolicity for the induced dynamics on the quotient space of center leaves such that we can follow the lines of the classical result and show that the induced dynamics on the quotient is actually conjugate to a hyperbolic toral diffeomorphism. \\
The difficulties arise at two places: first, the center holonomy may not be trivial so that the induced dynamics is certainly not hyperbolic (see example in \cite{BB12}), so the first step is to show that the holonomy group is trivial (if the unstable bundle is oriented). Second, the quotient space is a priori a topological manifold without a differentiable structure and the induced dynamics is not differentiable, so we pass usually to the manifold $M$ whenever differentiability is needed for the arguments. \\    
Important tools to circumvent these difficulties are the dynamical coherence for $f$ and the Shadowing Lemma for orbits of center leaves proved in \cite{BB12} which allows us to prove the central transitivity of $f$ (if the unstable bundle is one-dimensional) which quite directly implies the triviality of the center holonomy. Further, dynamical coherence means that center stable and center unstable foliations exist, and they can be projected to stable and unstable (topological) foliations on the quotient space. 

Let us now state the main result of this article: 
\begin{thm}
Let $f:M \rightarrow M$ be a partially hyperbolic $C^1$-diffeomorphism on a compact smooth connected manifold $M$. Assume that the center foliation $\mathcal{F}^c$ is an $f$-invariant uniformly compact foliation.\\
If $\dim E^u=1$ and $E^u$ is oriented, then the leaf space $M/\mathcal{F}^c$ is a torus $\mathbb{T}^q$ and the dynamics $F:M/\mathcal{F}^c \rightarrow M/\mathcal{F}^c$ induced by $f$ is topologically conjugate to an Anosov automorphism on $\mathbb{T}^q$ where $q = \codim \mathcal{F}^c$. 
\label{maintheorem}
\end{thm}
The proof of Theorem~\ref{maintheorem} consists of three main steps: 
\begin{enumerate}
\item \textit{Central transitivity}: we show that under the assumptions above the diffeomorphism $f$ has a center leaf whose orbit ist dense (Theorem~\ref{thm_transitive}). 
\item \textit{Trivial holonomy}: the center leaves have only trivial holonomy (Theorem~\ref{t.trivial}).
\item \textit{Margulis measure on leaf space}: we establish the existence of a Margulis measure on the leaf space (Theorem~\ref{theorem_measure}). 
\end{enumerate}
\begin{rem}
If $E^u$ is not oriented, then the partially hyperbolic diffeomorphism $f$ can be lifted to the orientation cover of $E^u$. The center stable foliation lifts cano\-nically to a transversely oriented codimension one foliation. Hence, the condition of orientability is always fulfillable.  
\end{rem}
In fact, we establish the third step without restriction on the dimension of the unstable bundle the existence of a Margulis measure, a family of measures whose support lies on unstable leaves, so we state this result here in its full generality:
\begin{thm}[Existence of Margulis measure]
Let $f: M \rightarrow M$ be a $C^2$ partially hyperbolic diffeomorphism with an $f$-invariant compact center foliation with trivial holonomy. Suppose $f$ is centrally transitive. Then there exists a family $\left\{\mu_{L^u}\right\}_{L^u \in \mathcal{F}^u}$ of measures with the following properties: 
\begin{enumerate} 
\item For every $L^u \in \mathcal{F}^u$ the measure $\mu_{L^u}$ is a regular Borel measure on the unstable leaf $L^u$, and the measure is positive on non-empty open sets within $L^u$. 
\item The family of measures is invariant under the center stable holonomy.
\item The family of measures is uniformly expanding under $f$, i.e. there exists $\beta > 1$ such that $\mu{f(L^u)}(f(A)) = \beta\mu_{L^u}(A)$ for every open set $A \subset L^u$.
\end{enumerate}
\label{theorem_measure}
\end{thm}
As a direct corollary - following the classical construction by Margulis in \cite{M70} - we obtain a complete $f$-invariant measure $\mu$ on $M$ defined on $\mathcal{F}^c$-saturated sets.
\begin{corol}\label{c.measure}
There exists an $f$-invariant Margulis measure $\mu$ on $M$ defined on $\mathcal{F}^c$-saturated sets.
\end{corol} 
\begin{rem}[Why codimension one?]
\begin{itemize}
\item As mentioned before, the proof of Theorem~\ref{maintheorem} depends on a result about codimension-one Anosov diffeomorphisms. Consequently, Theorem \ref{maintheorem} cannot be extended to a system with higher-dimensional unstable bundle without essential changes of the proof. It should be remarked that a corresponding generalization in the case of Anosov diffeomorphisms that they are always transitive and conjugate to an algebraic automorphism is only conjectured and not yet proved. 
\item Another difficulty arises at the following point: a finite holonomy is a continuous group action of a finite group of periodic homeomorphisms on a smoothly embedded disk in the manifold $M$ which is a priori not linear. We do not know a lot about these so-called continuous transformation groups if the dimension of the disk is greater than $2$. 
\item Further, the following easy examples illustrate that in the case of a two-dimensional unstable bundle a $2$-cover is not enough to eliminate center leaves with non-trivial holonomy. Besides, there can exist $f$-invariant submanifolds in $M$ which consist of center leaves with non-trivial holonomy. 
\end{itemize}
\end{rem}
\begin{example}
\begin{enumerate}
\item Let $A \in \Sl(2,\mathbb{Z})$ be a hyperbolic matrix with eigenvalues $0 < \lambda < 1 < \mu$. Define the following matrix $F_A$ by $\diag(A,A,\id) \in \Gl(5,\mathbb{Z})$ which induces a partially hyperbolic toral automorphism $f_A$ on $\mathbb{T}^5$ with a one-dimensional compact center foliation $\left\{x\right\} \times \mathbb{S}^1$. Let $B:=\begin{pmatrix} 0 & -\id& 0\\\id & 0& 0\\0 & 0 & 1\end{pmatrix} \in \Gl(5,\mathbb{Z})$. We define a free action of $\mathbb{Z}_4$ on $\mathbb{T}^5$ by  $H:= B + \begin{pmatrix}0&0&0&0&\frac{1}{4}\end{pmatrix}^T$. Then $f_A$ commutes with this finite group action and $\mathbb{T}^5 / Hx \sim x$ is a manifold with a partially hyperbolic diffeomorphism $F_A$ induced by $f_A$. The leaf space of the center foliation is a $4$-orbifold with four singular points, $$\left(\frac{1}{2},\frac{1}{2},\frac{1}{2},\frac{1}{2}\right), \left(\frac{1}{2},0,\frac{1}{2},0\right), \left(0,\frac{1}{2},0,\frac{1}{2}\right), \left(0,0,0,0\right)$$ with non-trivial holonomy $H$ of order four and twelve singular points of order two.   
\item Let $F_A \in \Sl(5,\mathbb{Z})$ be the matrix defined above which induces a partially hyperbolic torus automorphism $f_A$ on $\mathbb{T}^5$ with a one-dimensional compact center foliation $\left\{x\right\} \times \mathbb{S}^1$. Let $B:=\diag(\id,-\id,1) \in \Gl(5,\mathbb{Z})$ and define a free action of $\mathbb{Z}_2$ on $\mathbb{T}^5$ by $H:=B  + \begin{pmatrix}0&0&0&0& \frac{1}{2}\end{pmatrix}^T$. Then $f_A$ commutes with this group action and the resulting manifold $\mathbb{T}^5 / Hx \sim x$ carries a partially hyperbolic diffeomorphism $F_A$ induced by $f_A$. The leaf space of the center foliation is a $4$-orbifold with four tori of singular points: $$\left\{(x_1,x_2,a,b)^T\;\big| x_1,x_2 \in \mathbb{R}/\mathbb{Z}\right\}, \quad a,b=0,\frac{1}{2}.$$   
\end{enumerate}
\end{example}

One motivation for the topic of this article was the question if the study of partially hyperbolic diffeomorphisms with a uniformly compact center foliation could be restricted to the case of a compact center foliation with trivial holonomy. This would be possible if we show that any system with finite holonomy can be lifted to a finite cover where the holonomy vanishes everywhere. It should be remarked that a priori - without assuming a partially hyperbolic dynamics on the manifold - such a cover does not exist: there exist uniformly compact foliations whose quotient space cannot be finitely covered by a manifold, but is a so-called bad orbifold (see \cite{CH03}, section 7 for examples of such foliations which are called bad compact Hausdorff foliations in that article). So, to answer this question dynamical properties, features of compact foliations with finite holonomy and their interplay have to be utilized at the same time. A way to validate this intuition is to show that the center holonomy is globally affected whenever one locally builds in a center leaf with higher holonomy.\\ 
 
\textbf{Related work:} we would like to mention that Theorem~\ref{maintheorem} overlaps with a result by Gogolev in \cite{G11} where he proved a similar statement for partially hyperbolic diffeomorphisms with a uniformly compact center foliation of codimension 2 or with simply connected leaves and one-dimensional unstable direction. The assumption of simply connectedness directly implies trivial center holonomy.\\

\textbf{This work is organized as follows:} in the first section~\ref{s.background} we introduce and define the setting of this article and prove some preliminary results we need in the main proofs. In the subsequent sections the three steps leading to Theorem~\ref{maintheorem} are proved, every one in a separate section. Consequently, we start with the proof of the central transitivity (Theorem~\ref{thm_transitive}) in Section~\ref{s.transitive} where the first part is dedicated to the chain recurrent set and its decomposition (valid for any dimensions of the invariant foliations), and we go on with the proof of the trivial holonomy (Theorem~\ref{t.trivial}) of the center foliation under the assumption of a one-dimensional oriented unstable subbundle. Section~\ref{s.margulis} is employed by the construction of a family of measures with support on the unstable leaves (Theorem~\ref{theorem_measure}). This family is essential for the proof of the conjugacy between the quotient dynamics and a hyperbolic torus automorphism in Section~\ref{s.torus}.

\renewcommand{\abstractname}{Acknowledgements}
\begin{abstract}
I am deeply indebted to Christian Bonatti for his time to discuss the results of this paper with me. Further, I would like to thank S\'ebastien Alvarez for discussions about the Margulis measures and his helpful comments in this aspect. Last, but not least, the detailed comments and suggestions of the referees helped a lot to improve the present article.  
\end{abstract}

\section{Background}\label{s.background}

In the following paragraphs we introduce the basic concepts and the setting of this article, and we motivate our assumptions of a uniformly compact center foliation.  
\begin{rem}
We assume a compact center foliation $\mathcal{F}^c$ and accordingly, every center leaf is an embedded $C^1$-manifold in $M$. The center bundle $E^c$ is a continuous subbundle of $TM$ as a consequence of the definition of partial hyperbolicity, and the center foliation is tangent to this continuous distribution. Whenever we say \emph{foliation} in this article, we mean such a foliation with $C^1$-leaves tangent to a continuous distribution.
\end{rem}

\subsection{Uniformly compact foliations.}

We shortly recall the notion of holonomy that we use in this article: we consider a closed path $\gamma: [0,1] \rightarrow L$ with $\gamma(0)=x \in L$ which lies entirely inside one leaf $L$. 
We define a homeomorphism $H_{\gamma}$ on a smooth disk $T$ of dimension $q=\codim \mathcal{F}$ transversely embedded to the foliation $\mathcal{F}$ at $x$ which fixes $x$ and maps intersection points of a nearby leaf $L'$ onto each other following the path $\gamma$. We call it the \textit{holonomy homeomorphism along the path $\gamma$}. The definition of $H_{\gamma}$ (class of germs of $H_{\gamma}$) only depends on the homotopy class $[\gamma]$ of $\gamma$. Hence, we obtain a group homomorphism 
$$\pi_1\left(L,x\right) \rightarrow \Homeo\left(\mathbb{R}^q,0\right),$$
where $\Homeo(\mathbb{R}^q,0)$ denotes the classes of germs of homeomorphisms $\mathbb{R}^q \rightarrow \mathbb{R}^q$ which fix the origin.
The image of this group homomorphism is called the \textit{holonomy group} of the leaf $L$ and denoted by $\Hol(L,x)$. 
By taking the isomorphism class of this group it does not depend on the original embedding of $T$ in $M$. It is easily seen that any simply connected leaf has a trivial holonomy group. A classical result (proved in \cite{H77} and \cite{EMT77}) states that leaves with trivial holonomy group are generic.\\  
We say that a foliation has \textit{finite holonomy} if the holonomy group $\Hol(L,x)$ for any $x \in M$ is a finite group. Consequently, every holonomy homeomorphism is periodic. We call a foliation whose leaves are all compact and have finite holonomy group a \textit{uniformly compact foliation}. These foliations are also called \emph{compact Hausdorff foliations} in foliation theory referring to the fact that the quotient space of such a foliation is Hausdorff. This notion is also equivalent to the notion of a foliation whose leaves have \emph{uniformly bounded volume}.

\begin{example}[Partially hyperbolic diffeomorphisms with non-trivial center holonomy]
Bonatti and Wilkinson construct in \cite{BoW05} a partially hyperbolic diffeomorphism on the 3-torus fibering over the 2-torus which commutes with a $\mathbb{Z}_2$-action on the 3-torus. Hence, the system induces a partially hyperbolic diffeomorphism on the resulting quotient manifold with a center foliation by circles and four center leaves with non-trivial holonomy corresponding to the fixed points of the $\mathbb{Z}_2$-action on the 2-torus and coinciding with the singular leaves of the Seifert fibration over the 2-orbifold. 
\label{example_bundle}
\end{example}
This example can be generalized in the following way:
\begin{example}
Let $M$ be a compact manifold, $G$ a compact Lie group and $H < G$ a finite subgroup which acts freely on $M \times G$ by the canonical product action $(x,g) \mapsto (h.x, gh^{-1})$. Let $f: M \rightarrow M$ be a $H$-equivariant Anosov diffeomorphism, i.e. $f(h.x)=h.f(x)$ for all $x \in M$ and $h \in H$. Let $\phi: M \rightarrow G$ be a smooth $H$-invariant map. 
Then $N:=M\times_H G= M \times G /(h.x,g) \sim (x,gh^{-1})$ is a compact manifold and the partially hyperbolic skew product $F_{\phi}: M \times G \rightarrow M \times G$, $F_{\phi}(x,g)=(f(x),\phi(x)g)$ induces a partially hyperbolic diffeomorphism $f_{\phi}:N \rightarrow N$. The center foliation $\left\{\left\{x\right\} \times G/ (x \sim h.x)\right\}$ is a compact foliation and the holonomy group of every center leaf $L_x=G/H_x$ through $x$ equals the isotropy group $H_x$ of the $H$-action on $M$.
\end{example}

Although all leaves are compact manifolds, it can occur that the volume of a ball with increasing radius inside a leaf is unbounded. This cannot occur if the codimension of the foliation is one or two, as shown in \cite{R52} for codimension one on non-compact manifolds and in \cite{E72} for codimension two on compact manifolds. But Sullivan constructed in \cite{S76a} and \cite{S76b} a flow on a compact $5$-manifold, that is a codimension 4 foliation, such that every orbit is periodic but the length of orbits is unbounded. Later, Epstein and Vogt in \cite{EV78} constructed an example of a foliation by circles with the same properties on a compact $4$-manifold. In the following we will therefore assume a uniformly compact center foliation if the codimension is strictly greater than 2.\\
The reason why we restrict ourselves to uniformly compact foliations is closely related to the point that in that case there exist small saturated neighborhoods $V(L)$ of any leaf $L$ that are foliated bundles $p:V(L) \rightarrow L$ and the holonomy group is actually a group. This fact is the content of the Reeb Stability Theorem which is proved in the present form in \cite{CC00}:    
\begin{otherthm}[Generalized Reeb Stability]
If $L$ is a compact leaf of a foliation $\mathcal{F}$ on a manifold $M$ and if the holonomy group $\Hol(L,y)$ is finite, then there is a normal neighborhood $p:V~\rightarrow~L$ of $L$ in $M$ such that $\left(V, \mathcal{F}|_V,p\right)$ is a fiber bundle with finite structure group $\Hol(L,y)$. Furthermore, each leaf $L'|_V$ is a covering space $p|_{L'}: L'~\rightarrow~L$ with $k~\leq~\left|\Hol(L,y)\right|$ sheets and the leaf $L'$ has a finite holonomy group of order $\frac{\left|\Hol(L,y)\right|}{k}$. 
\label{thm1}
\end{otherthm}
Not only locally, but also globally the leaf space of a uniformly compact foliation has a good structure, as the quotient space of the foliation is Hausdorff (see \cite{E76} and \cite{M75}): 
\begin{otherthm}[Epstein, Millet]
Let $M$ be a foliated space with each leaf compact. Then the following conditions are equivalent: 
\begin{itemize}
 \item The quotient map $\pi: M \rightarrow M/\mathcal{F}$ is closed. 
\item Each leaf has arbitrarily small saturated neighborhoods. 
\item The leaf space $M/\mathcal{F}$ is Hausdorff. 
\item If $K \subset M$ is compact then the saturation $\pi^{-1}\pi K$ of $K$ is compact, this means, the set of leaves meeting a compact set is compact.
 \item The holonomy group on every leaf is finite.  
\end{itemize}
\label{theorem_epstein}
\end{otherthm}
Further, it can be elementarily proved that the quotient topology is generated by the Hausdorff metric $d_H$ between center leaves in $M$ and hence, $M/\mathcal{F}^c$ is a compact metric space. 
\begin{rem}
A natural question is if there can exist a partially hyperbolic diffeomorphism with an $f$-invariant compact center foliation which is not uniformly compact. There are results which answer this question negatively under additional assumptions by Gogolev in \cite{G11} and Carrasco in \cite{C10}. 
\end{rem}
\begin{rem}
In the case of a trivial center holonomy the resulting leaf space is a topological manifold. We showed in \cite{BB12} that in this case the partially hyperbolic diffeomorphism is dynamically coherent, and accordingly, the center stable and center unstable foliations induce stable and unstable topological foliations in the leaf space, the induced homeomorphism $F$ on the leaf space is then expansive and has a local product structure. 
\end{rem}

\subsection{Finite homeomorphism groups.}
We assume that the holonomy group of any center leaf is a finite group, unless in the case of a codimension-$2$ center foliation where it is implied by the compactness of the foliation. For every $x \in M$ and $L \in \mathcal{F}^c$ the holonomy group $\Hol(L,x)$ consists of periodic homeomorphisms and acts continuously on the smooth manifold $T$ transverse to $L$ at $x \in L$. Therefore, we can find some useful results on the fixed point sets of such homeomorphisms from the theory of finite transformation groups which we need and cite here. First, it is quite obvious that any non-trivial periodic homeomorphism of $(-1,1) \rightarrow (-1,1)$ which leaves the origin fixed reverses the orientation.  
\begin{lemma}
Let $\Homeo\left((-1,1),0\right)$ be the group of germs at $0 \in \mathbb{R}$ of homeomorphisms which leave $0 \in \mathbb{R}$ fixed. Let $G$ be a finite subgroup of $\Homeo\left((-1,1),0\right)$. Then $G$ has at most $2$ elements. If $G$ has $2$ elements, then one of them reverses orientation. 
\label{lemma_1dim}
\end{lemma}
\begin{rem}
As a direct implication it follows that every transversely oriented codimension one foliation with finite holonomy has actually trivial holonomy. 
\end{rem}
The following theorem describes the fixed point set of periodic homeomorphisms of the two-dimensional disk $\mathbb{D}^2 \subset \mathbb{R}^2$ (announced in \cite{K19} and \cite{B19}, proved in \cite{E34}.) 
\begin{otherthm}[Theorem of Ker\'{e}kj\'{a}rt\'{o}]
Suppose $g: \mathbb{D}^2 \rightarrow \mathbb{D}^2$ is a periodic homeomorphism of period $n > 1$. Then $g$ is topologically conjugate to a orthogonal matrix $A \in O(2)$, i.e. there is a homeomorphism $h$ such that $g=h^{-1}Ah$. If $g$ is orientation-preserving, the set of fixed points is a single point which is not on the boundary. If $g$ is orientation-reversing, then $g^2=\id$ and the set of fixed points is a simple arc which divides $\mathbb{D}^2$ into two topological discs which are permuted by $g$. 
\label{kerekjarto}
\end{otherthm}
We will need that any periodic homeomorphism on a open disk is conjugate to a finite order rotation around the origin or to a reflection. This is a direct consequence of the Theorem of Ker\'{e}kj\'{a}rt\'{o}: 
\begin{corol}[\cite{CK94}]\label{c.rotation}
Let $f: \mathbb{R}^2 \rightarrow \mathbb{R}^2$ be a periodic homeomorphism. Then $f$ is conjugate to a finite order rotation around the origin or to the reflection about the $x$-axis.
\end{corol}
Unfortunately, equivalent statements in higher dimensions are only available in the case of a differentiable group action: Bochner showed in \cite{B45}) that if a compact group of local transformations has a fixed point and if it is uniformly smooth ($C^k, k \geq 1$), then there exists a smooth local change of coordinates such that the transformations act linearly with respect to this new coordinate system. In particular, if a uniformly compact foliation of codimension $q$ is transversely differentiable, the holonomy group can be assumed to act linearly as a subgroup of $O(q)$ on the transversal.\\
But Bing constructed in \cite{B52} a counterexample of a $2$-periodic homeomorphism $h$ of the $3$-dimensional euclidean space (or on the $3$-sphere) which has a double horned sphere as fixed point set and is therefore not topologically conjugate to a linear map. With this involution one can easily construct a foliation defined by a fibration on a mapping torus $ \mathbb{S}^3\times [0,1] / (x,0) \sim (h(x),1)$ which has leaves (corresponding to the fixed points of $h$) with a holonomy group generated by this homeomorphism $h$.   
\begin{question}
Does there exist a partially hyperbolic diffeomorphism with a compact center foliation such that 
\begin{enumerate}
\item the center holonomy of every leaf is induced by a global group action and
\item the fixed point set of this global group action (corresponding to the center leaves with maximal holonomy) is not a submanifold?
\end{enumerate}
More generally, does there exist a partially hyperbolic diffeomorphism with a uniformly compact center foliation which has a holonomy group which is not a group of isometries (there does not exist a local coordinate system such that the group acts as a group of isometries with respect to these local coordinates)?  
\end{question}

\subsection{Dynamical coherence and structural stability}
Although there exist three foliations tangent to the center, unstable and stable bundle respectively it is not clear that  the center unstable bundle $E^{cu}:=E^c\oplus E^u$ or the center stable bundle $E^{cs}$ are integrable. A counterexample on the 3-torus is constructed in \cite{HHU10a}. The property in question is defined as follows: 
a partially hyperbolic diffeomorphism $f$ is called \textit{dynamically coherent}, if there exist a center stable foliation $\mathcal{F}^{cs}$ tangent to $E^{cs}=E^s \oplus E^c$ and a center unstable foliation $\mathcal{F}^{cu}$ tangent to $E^{cu}=E^{c} \oplus E^{u}$.\\
dynamical coherence is essential for our investigation of partially hyperbolic systems with an $f$-invariant uniformly compact center foliation. Due to dynamical coherence any smooth manifold transverse to a center leaf $L$ is foliated by stable and unstable leaves induced by the center stable and center unstable foliation.\\
We showed in \cite{BB12} that every partially hyperbolic diffeomorphism with an $f$-invariant uniformly compact center foliation is dynamically coherent. 
Further, using the shadowing property for center leaves we showed in \cite{BB12} that $(f,\mathcal{F}^c)$ is plaque expansive and hence, structurally stable (in the sense of \cite{HPS70}), i.e. every partially hyperbolic diffeomorphism $g$, sufficiently close to $f$ with respect to the $C^1$-topology, has a center foliation which is leafwise topologically conjugate to the center foliation of $f$. This means, we can perturb $f$ a bit in order to get a diffeomorphism as smooth as we want without changing the dynamical properties. In order to utilize the absolute continuity of the stable and unstable foliations we have to use this fact later on to prove Theorem~\ref{theorem_measure}.   

\subsection{Codimension 2 and Codimension 3}\label{s.codimension}
Before we start with the proofs of our results we outline the proof of Theorem~\ref{maintheorem} for the case that the center foliation has codimension two or three. In these cases the low codimension allow a more direct approach:  
\subsubsection{Codimension 2:} We recall that we do not have to assume a finite holonomy in this case as it is implied by the compactness of leaves (cf. \cite{E72}). As $\dim E^u = \dim E^s=1$ the center holonomy is trivial if the stable and unstable bundles are oriented. So, by lifting the diffeomorphism to the at most $4$-sheeted orientation cover of $M$ we can assume that the center foliation is without holonomy. The $2$-dimensional leaf space $M/\mathcal{F}^c$ is then a topological manifold and the partially hyperbolic diffeomorphism $f$ induces an expansive homeomorphism $F$ on $M/\mathcal{F}^c$. A theorem by Lewowicz in \cite{L89} implies then that $F$ is conjugate to a hyperbolic toral automorphism. The proof that a $2$-sheeted cover suffices to eliminate the center holonomy follows the lines of Section~\ref{s.trivial} (see proof in Subsection~\ref{proof_codimension2}). Further, the result is more precise in this case: we can easily conclude that there are either four center leaves with non-trivial holonomy or none. Every non-trivial holonomy homeomorphism is conjugate to a rotation by $\pi$. The leaf space is either a $2$-orbifold with four elliptic points and underlying manifold a $2$-sphere or a $2$-torus.  
\subsubsection{Codimension 3:} Assuming that $\dim E^s=2$ and $\dim E^u=1$ we can also completely describe the possible center holonomy groups as we can use the Corollary~\ref{c.rotation} above. So, if $E^u$ and $E^s$ are oriented, we know that every center holonomy homeomorphism is the identity in the unstable direction and a finite rotation in the stable direction. It is left to rule out the possibility of a non-trivial rotation, this is done with the methods of Section~\ref{s.trivial}. Consequently, we have trivial center holonomy and the leaf space is a topological $3$-manifold with an induced expansive homeomorphism $F$. Utilizing the central transitivity of $f$, proved below, a result by Vieitez in \cite{V00} (together with \cite{F70}) implies directly that $F$ is conjugate to a hyperbolic toral automorphism. In contrary to codimension 2, it is indispensable for this conclusion that $f$ is centrally transitive (see Subsection~\ref{proof_codimension3} below). As in the case above, we get a complete description of the possible partially hyperbolic systems with a compact center foliation of codimension $3$: the leaf space $M/\mathcal{W}^c$ is either a $3$-orbifold $\mathbb{T}^3/\left(-\id\right)$ with $8$ singular points or a $3$-torus.

\section{The chain recurrent set of the quotient dynamics}\label{s.chains}
Let $f$ be a partially hyperbolic $C^1$-diffeomorphism with a uniformly compact center foliation $\mathcal{F}^c$. Denote by $F$ the homeomorphism on the quotient space $M/\mathcal{F}^c$ induced by $f$. We call the dynamics by $F$ on the leaf space $M/\mathcal{F}^c$ shortly the \emph{quotient dynamics}. The results in this subsection hold in this general setting without any restrictions on the codimension of the center foliation. For any $\epsilon > 0$ we call an \emph{$\epsilon$-chain} of center leaves from $L$ to $L'$ a finite sequence $\left\{L_i\right\}_{i=0}^n$ of center leaves such that $d_H(f(L_i),L_{i+1}) < \epsilon$ for $i=0,\dots,n-1$ and $L_0=L$ and $L_n=L'$. A leaf $L$ is \emph{chain recurrent} if for all $\epsilon > 0$ there is an $\epsilon$-chain from $L$ to $L$. Denote the set of chain recurrent leaves by $\Cr(F)$ called \emph{chain recurrent set}. We denote by 
$$\Cr(L):=\left\{L' \in \mathcal{F}^c\,\big|\; \forall \epsilon > 0 \; \exists \;\epsilon-\mbox{chain}\;\mbox{from}\; L\;\mbox{to}\; L' \;\mbox{and one from}\; L' \;\mbox{to}\; L\;\right\}$$
the \emph{chain recurrent class} of a center leaf $L$. Then the chain-recurrent set $\Cr(F)$ is a union $\bigcup \Cr(L)$ of chain-recurrent classes. Each chain recurrent class is a compact $\mathcal{F}^c$-saturated and $f$-invariant set. For $L \in \mathcal{F}^c$ we denote by $\mathcal{F}^s_{loc}(L)$ the union of local stable leaves through the center leaf $L$, i.e. $\bigcup_{z \in L} L^s_{loc}(z)$, and equivalently with $\mathcal{F}^u_{loc}(L)=\bigcup_{z\in L}L^u_{loc}(z)$. \\
We can then prove the following results about the chain recurrent classes of $F$:
\begin{lemma}
The induced homeomorphism $F: M/\mathcal{F}^c \rightarrow M/\mathcal{F}^c$ on the leaf space has only finitely many chain-recurrent classes. 
\label{lemma_chain_recurrent}
\end{lemma}
\begin{proof}[Proof of Lemma~\ref{lemma_chain_recurrent}]
Assume that there are infinitely many chain-recurrent classes. Then for every $\delta > 0$ there exist center leaves $L_1$ and $L_2$ such that $\Cr(L_1)\cap \Cr(L_2) = \emptyset$ and $L_2 \subset B_{H}(L_1,\delta)$. We choose $\delta > 0$ sufficiently small. Then the intersection of $\mathcal{F}^u_{loc}(L_1)$ with $\mathcal{F}^s_{loc}(L_2)$ and conversely of $\mathcal{F}^s_{loc}(L_1)$ with $\mathcal{F}^u_{loc}(L_2)$ contain at least one point, and we can construct for every $\epsilon > 0$ an $\epsilon$-chain which connect $\Cr(L_1)$ and $\Cr(L_2)$ and conversely.  
This implies the equality of both chain recurrent classes: $\Cr(L_1)=\Cr(L_2)$.  

\end{proof}
We denote the finite chain recurrent classes by $\Omega_1, \dots, \Omega_n$. We have $\bigcup_{i=1}^n\Omega_i~=~\Cr(F)$. For all $L \in \mathcal{F}^c$ define the compact sets $\omega(L)=\bigcap_{n \geq 0}\overline{\left(\bigcup_{k \geq n}f^k(L)\right)}$ and $\alpha(L)=\bigcap_{n \geq 0}\overline{\left(\bigcup_{k \geq n}f^{-k}(L)\right)}$, called \emph{$\omega$- and $\alpha$-limit sets}. 
First, we recall the following fact: 
\begin{claim}\label{c.limitset}
For all $L \in \mathcal{F}^c$ there exist $i,j \in \left\{1,\dots,n\right\}$ such that $\omega(L) \subset \Omega_i$ and $\alpha(L) \subset \Omega_j$. 
\end{claim}
The finite number of chain recurrent classes $\Omega_1, \dots, \Omega_n$ comprises the chain recurrent set $\Cr(F)$. There exists a natural way to order the chain recurrent classes: We say that $\Omega_i \leq \Omega_j$ if for every $\epsilon > 0$ there exists an $\epsilon$-pseudo orbit $\left\{L_n\right\}_{n \in \mathbb{Z}}$ of center leaves such that $d_H(L_n,\Omega_i) \rightarrow 0$ for $n \rightarrow -\infty$ and $d_H(L_n, \Omega_j) \rightarrow 0$ for $n \rightarrow \infty$.  
This order on the set of chain recurrent classes together with the finiteness of chain recurrent classes implies directly the following corollary, due to Conley theory (see \cite{C78}): 
\begin{corol}
There exists a chain-recurrent class $R \subset \Cr(F)$ which is a repeller, i.e. there exists a $\mathcal{F}^c$-saturated neighborhood $U \supset R$ such that $f^{-1}(\overline{U}) \subset U$ and $\bigcap_{n \geq 0}f^{-n}U=R$. 
\end{corol}
We recall that we proved in \cite{BB12} the dynamical coherence of $f$ and further the so-called completeness, that is, every center stable leaf $L^{cs}$ of the center stable foliation $\mathcal{F}^{cs}$ tangent to $E^c \oplus E^s$ is equal to the union of stable leaves through one center leaf $L \subset L^{cs}$, what we denote by the stable leaf $\mathcal{F}^s(L)$ of the center leaf $L$. 
With this notation, we state the following property which plays a crucial role in the proof of transitivity: 
\begin{lemma}\label{l.stablemanifold}
$$\bigcup_{L \in \Cr(F)}\mathcal{F}^s(L)=M.$$
\end{lemma}
\begin{proof}[Proof of Lemma~\ref{l.stablemanifold}]
Let $L \in \mathcal{F}^c$ be an arbitrary center leaf. Claim~\ref{c.limitset} implies that there exists a chain recurrent class $\Omega_i$ such that $\omega(L)\subset \Omega_i$. Consequently, there exist $n > 0$ and $L_n \subset\Omega_i$ such that $f^nL$ and $L_n$ are sufficiently close such that there is a center leaf $L'_n \subset \mathcal{F}^u_{loc}(L_n) \cap \mathcal{F}^s_{loc}(f^nL)$ contained in the intersection of the center stable with the center unstable leaf. We prove that $L'_n\subset \Omega_i$: 
\begin{claim}
With the notations above we have $L'_n \subset \Omega_i$. 
\label{c.stablemanifold}
\end{claim}  
\begin{proof}[Proof of Claim~\ref{c.stablemanifold}]
We have $\alpha(L'_n) = \alpha(L_n)$ because $L'_n \subset \mathcal{F}^u_{loc}(L_n)$, and therefore $\alpha(L'_n) \subset \Omega_i$. On the other hand, $L'_n \subset \mathcal{F}^s_{loc}(f^nL)$, so $\omega(L'_n)=\omega(L) \subset \Omega_i$. As $L'_n$ is chain recurrent to $\alpha(L'_n)$ and $\omega(L'_n)$, this implies that $L'_n \subset \Omega_i$. 
\end{proof}
This implies directly that $f^nL \subset \mathcal{F}^s(L'_n)$ with $L'_n \subset \Omega_i$. Hence, $L \subset \mathcal{F}^s(\Omega_i)$. 
\end{proof}
Because every center stable leaf $L^{cs} \in \mathcal{F}^{cs}$ is by the completeness of the foliations equal to the stable foliation $\mathcal{F}^s(L)$ through any center leaf $L \subset L^{cs}$ we get the following consequence of the previous Lemma: 
\begin{corol}\label{corol_intersection}
For every center stable leaf $L^{cs}$ there exists $i\in \left\{1,\dots,n\right\}$ such that $L^{cs}\cap \Omega_i \neq \emptyset$. 
\end{corol}
With the help of the previous results we can now prove the transitivity of $F$ under the additional hypothesis of a one-dimensional unstable direction. 

\section{Central transitivity}\label{s.transitive}

We proceed with the proof of the central transitivity, the first step in order to prove Theorem~\ref{maintheorem}. \\
The central transitivity of $f$ is equivalent to the transitivity of the induced homeomorphism $F$ on the leaf space which is a compact metric space with induced (probably) singular stable and unstable foliations. Recall that Newhouse in \cite{N70} proved that any codimension one Anosov diffeomorphism is transitive. We can not adapt his proof to our setting as it essentially requires differentiability and the existence of non-singular foliations which provide a local product structure. Nevertheless, we can use the main idea of a simplified proof by Hiraide in \cite{H01} of Newhouse's theorem. The adaptation of this idea is not straightforward as points are substituted by compact center manifolds with probably non-trivial holonomy. So we do not have a local product structure. In fact, a local stable leaf may intersect a local unstable leaf several times (cf. example in \cite{BoW05} and the discussion of its pathological behavior in \cite{BB12}). The main ingredient to accomplish the proof is the Shadowing Lemma proved in \cite{BB12} and the decomposition of the chain recurrent set (see Section~\ref{s.chains} above) in the setting of a uniformly compact center foliation. \\
This section is dedicated to the proof of the following theorem:
\begin{thm}[Central Transitivity]
Let $f: M \rightarrow M$ be a partially hyperbolic $C^1$-diffeomorphism with a uniformly compact $f$-invariant center foliation. Suppose that $\dim E^u =1$. Then $f$ is centrally transitive.
\label{thm_transitive}
\end{thm}
\begin{rem}\label{r.oriented}
First, we observe that Theorem \ref{thm_transitive} proved under the additional assumption of an orientable unstable bundle $E^u$ implies Theorem \ref{thm_transitive}. If the unstable bundle is non-orientable, the whole system $f$ can be lifted to a partially hyperbolic system $\tilde{f}$ on the $2$-cover $\tilde{M}$ of orientation of $E^u$ and $\tilde{f}$ fulfills all the assumptions of Theorem \ref{thm_transitive} plus orientability of $E^u$, hence, the diffeomorphism $\tilde{f}$ is centrally transitive and consequently, $f$ is centrally transitive.  
\end{rem}
\begin{rem}
In the following we assume that the codimension of the center foliation is greater or equal than three. The case of a \textbf{codimension 2 center foliation} can be treated in a far easier way without using central transitivity as we explained above in Subsection~\ref{s.codimension} and Subsection~\ref{proof_codimension2} respectively.
\end{rem}
From now on we suppose that $E^u$ is one-dimensional and oriented. As a consequence, the center-stable foliation $\mathcal{F}^{cs}$ has trivial holonomy. The proof of Theorem~\ref{thm_transitive} divides into the following two steps: 
\begin{enumerate}
\item We prove that $f$ is centrally chain transitive, i.e. that $F$ on the leaf space is chain transitive (Proposition~\ref{proposition_1}). 
\item Any partially partially hyperbolic diffeomorphism $f$ with an $f$-invariant uniformly compact center foliation which is centrally chain transitive is in fact centrally transitive. This is a direct implication of the Shadowing Lemma on the leaf space, proved in \cite{BB12}. 
\end{enumerate}
So, for Theorem~\ref{thm_transitive} it suffices to prove the following proposition: 
\begin{prop}
Under the assumptions of Theorem \ref{thm_transitive}, the diffeomorphism $f$ is centrally chain transitive, i.e. $F: M/\mathcal{F}^c \rightarrow M/\mathcal{F}^c$ is chain transitive. 
\label{proposition_1}
\end{prop}

\subsection{Proof of Proposition \ref{proposition_1}.} The notations used in this subsection are introduced in Section~\ref{s.chains} above. 
We consider a repeller $R \subset M$ and a $\mathcal{F}^c$-saturated neighborhood $U \supset \Omega$ such that $U \cap \Omega_j = \emptyset$ for any other chain-recurrent class $\Omega_j$. For all center leaves $L \in R$ it is $\mathcal{F}^{s}(L) \subset R$. 
For the proof of Proposition~\ref{proposition_1} we show the following Lemma: 
\begin{lemma}[Principal Lemma]
Assume there exists a repeller $R \neq M$. For any open $\mathcal{F}^c$-saturated neighborhood $U \supset R$ such that $U \cap \Cr(F)=R$ there exists a center leaf $L \subset U\setminus R$ such that $\mathcal{F}^{s}(L) \subset U$. 
\label{lemma_principal}
\end{lemma}
This Lemma together with the previous results about the chain recurrent classes in Section~\ref{s.chains} implies Proposition~\ref{proposition_1} directly:
\begin{proof}[Proof of Proposition~\ref{proposition_1}]
Lemma~\ref{lemma_principal} implies that there exists $L \subset U$ such that $\mathcal{F}^{s}(L) \subset U$. As $\bigcup_{L' \in \Cr(F)} \mathcal{F}^{s}(L') =M$ from Corollary~\ref{corol_intersection} the intersection $\mathcal{F}^{s}(L) \cap \Cr(F) \neq \emptyset$. Hence, there exists $j \in \left\{1,\dots,k\right\}$ and $L' \subset \Omega_j$ such that $\mathcal{F}^{s}(L)=\mathcal{F}^{s}(L')$. This implies that $L' \subset U \cap \Omega_{j}$ contradicting the assumption that $U \cap \Cr(F)=R$. So we get $R=M$. \end{proof}
We prove now the Principal Lemma \ref{lemma_principal} in several smaller steps: 
\begin{lemma}\label{l.cover} Let $f: M \rightarrow M$ be a partially hyperbolic $C^1$-diffeomorphism with an $f$-invariant uniformly compact center foliation $\mathcal{F}^c$. Assume $\dim E^u=1$ and $E^u$ is oriented. Further, let $R \neq M$ be an $\mathcal{F}^c$-saturated repeller with respect to the dynamics $F$ of the center foliation. Then there exist finitely many $\mathcal{F}^c$-saturated neighborhoods $V_i \subset U$ with $\bigcup \interior{V_i} \supset R$ such that 
\begin{itemize}
\item there is a homeomorphism $\phi_i: V_i \rightarrow U_i \times [-1,1]$ where
\item  $U_i$ is a tubular neighborhood of $L_i \in \mathcal{F}^c$ inside $\mathcal{F}^s(L_i)$ and $p_i: U_i \rightarrow L_i$ fibers over $L_i$ and
\item $\phi^{-1}_i(\left\{x\right\} \times (-1,1))$ coincides with a unstable segment of the unstable foliation.
\end{itemize}
Further, every border plaque homeomorphic to $U_i \times \left\{\pm 1\right\}$ is either contained in the interior of $R$ or it lies completely outside of $R$. In particular, there exists $i$ such that the plaque $U_i \times \left\{\alpha\right\}$, $\alpha \in \left\{-1,1\right\}$ is disjoint from $R$.   
\end{lemma}
\begin{proof}[Proof of Lemma~\ref{l.cover}]
Under the assumptions above we showed that $f$ is dynamically coherent, hence, the center stable foliation is subfoliated by the stable and center foliation. Therefore, we can apply the Reeb Stability Theorem to the uniformly compact foliation $\mathcal{F}^c$ inside the center stable foliation $\mathcal{F}^{cs}$. This together with the compactness of the repeller $R$ gives us directly the existence of finitely many open $\mathcal{F}^c$-saturated sets $W_j \subset U$ inside the center stable foliation $\mathcal{F}^{cs}$ through $R$ such that every set $W_j$ is a tubular neighborhood of a center leaf $L_j$ and $p_j: W_j \rightarrow L_j$ is a fiber bundle where every fiber is a stable disk inside $\mathcal{F}^s(L_j)\cap W_j$.\\
The unstable foliation $\mathcal{F}^u$ is one-dimensional, oriented and transverse to the center stable foliation $\mathcal{F}^{cs}$. The center unstable foliation is thanks to the dynamical coherence subfoliated by the unstable and center foliation. By applying the Reeb Stability Theorem to the uniformly compact center foliation inside the center unstable foliation we get $\mathcal{F}^c$-saturated tubular neighborhoods $Y_k$ such that $Y_k$ is homeomorphic to $L_k \times (-1,1)$ where $L_k$ is a center leaf inside $Y_k$. Every $Y_k$ is trivially foliated as a product because the holonomy along center leaves is trivial. We fit both sets of $\mathcal{F}^c$-tubular neighborhoods $(W_j)$ and $(Y_k)$ together by centering them in the same center leaf $L_k$ of $Y_k$ and by shrinking them if necessary such that we obtain some compact $\mathcal{F}^c$-saturated sets $V_i$ which are homeomorphic to $U_i \times [-1,1]$ where every $U_i$ has the same properties as $W_j$ above. This gives the first three items of the Lemma.    \\
The last item is a direct consequence of the fact that $R$ is a repeller and $\mathcal{F}^{cs}$-saturated. Therefore any tubular neighborhood $U_i$ inside a center stable leaf lies either completely inside or outside of $R$.  \\
We have now a cover with compact sets of $R$ inside the open neighborhood $U$ of $R$. By changing a bit the parametrization of the unstable foliation of the compact sets $V_i$ we could assure that there exists one $V_i$ such that either $U_i \times \left\{1\right\}$ or $U_i \times \left\{-1\right\}$ lies completely outside $R$.  
 
\end{proof}
\begin{rem}\begin{enumerate}
\item For simplicity, we denote the $\mathcal{F}^c$-saturated stable plaque through $x \in U$ by $U_x$ abbreviating $\phi_i^{-1}(U_i \times \left\{t\right\})$. We also identify $V_i$ with its homeomorphic image $U_i \times [-1,1]$. 
\item As every unstable leaf of $\mathcal{F}^u$ is homeomorphic to $\mathbb{R}$ we have a natural order on every unstable segment. 
\item By changing the orientation of $\mathcal{F}^u$ we can assume that there exists $i$ such that $U_i \times \left\{1\right\}$ lies outside $R$. 
\item By $D^s(x,\gamma)$ we denote a stable disk (inside a stable leaf) of diameter $\gamma$ centered at $x$. 
\end{enumerate}
\end{rem}  
From now on we fix a cover $(V_i)$ of $R$ with the properties given by Lemma~\ref{l.cover}.
In every neighborhood $U_i \times\left[-1,1\right]$ if possible we choose points $m_i \in R$ which are the maximal points of intersection with respect to the orientation of $E^u$, i.e. the set $L^u_{+}(m_i)=\left\{z \in L^u(m_i)\;\big|\; z > m_i\right\}$ does not intersect $R$.
\begin{lemma}\label{l.delta}
There exists a compact set $\Delta \subset L^{cs}(m_1)$ such that 
\begin{enumerate}
\item $\Delta$ contains all plaques $U_{m_j}$ of points $m_j$ where $m_j \in L^{cs}(m_1)$ and
\item $L^{cs}(m_1) \setminus \Delta$ is path-connected.
\end{enumerate}
\end{lemma}

To find a compact set with the first property of Lemma~\ref{l.delta} is an easy task, so the main problem is to find a compact set whose complement is path connected.  
Consequently, the proof of Lemma~\ref{l.delta} relies on the following proposition: 
\begin{prop}\label{p.connected}
For every center stable leaf $L^{cs} \in \mathcal{F}^{cs}$ and for every compact set $K \subset L^{cs}$ there exists a compact set $N \subset L^{cs}$ such that $K \subset N$ and $L^{cs}\setminus N$ is connected. 
\end{prop}
Proposition~\ref{p.connected} is proved by the following three Lemmata which relie on the following fact: 
the center foliation $\mathcal{F}^c$ is a uniformly compact foliation and the manifold $M$ is compact, hence, there exists a maximal cardinality of the holonomy groups of the center leaves. Consequently, there exists for every center stable leaf $L^{cs} \in \mathcal{F}^{cs}$ a center leaf $L_0 \subset L^{cs}$ with maximal holonomy inside $L^{cs}$, this means, for every center leaf $L_1 \subset L^{cs}$ it holds that $|\Hol(L_1)| \leq |\Hol(L_0)|$. 
These considerations allow us to prove the following Lemma:
\begin{lemma}\label{l.maximal2}
Let $L^{cs} \in \mathcal{F}^{cs}$ be a center stable leaf and let $L_0 \in L^{cs}$ be a center leaf such that $|\Hol(L_1)| \leq |\Hol(L_0)|$ for every center leaf $L_1 \subset L^{cs}$. Then $L_0$ intersects every stable leaf $L^s \subset L^{cs}$ exactly once. 
\end{lemma}
\begin{proof}[Proof of Lemma~\ref{l.maximal2}]
We have fixed a finite cover $(V_i)$ of $R$ defined by Lemma~\ref{l.cover}. Every $\mathcal{F}^c$-saturated set $V_i$ is homeomorphic to $U_i \times [-1,1]$. Then there exists a Lebesgue number $\gamma > 0$ such that for all $x \in R$ there exists $i$ such that $\mathcal{F}^c(D^s(x,\gamma)) =\left\{L_y \in \mathcal{F}^c\,\big|\; y \in L^s_{\gamma}(x)\right\} \subset U_i$. \\
Assume there exists $L^s \subset L^{cs}$ such that there are $x_1,x_2 \in L^s \cap L_0$ with $x_1 \neq x_2$. Then we could assume that $x_1,x_2$ lie inside a local stable manifold $L^s_{\gamma}(x_1)$ with radius $\gamma$. Consequently, there exists $i$ such that $\mathcal{F}^c(D^s(x_1,\gamma)) \subset U_i$ and therefore this implies $L_0 \subset U_i$. The set $U_i$ is a fiber bundle over a center leaf with locally maximal holonomy group, this means, there exists a center leaf $L_i \subset U_i$ such that $p_i: U_i \rightarrow L_i$ is a fiber bundle and $p_i|_{L_0}: L_0 \rightarrow L_i$ is a finite covering. As $L_0$ intersects $D^s(x_1,\gamma)$ at least in two different points $x_1,x_2$ there exists $x \in L_i$ such that $x_1,x_2 \in p_i^{-1}|_{L_0}(x)$. This implies that $2|\Hol(L_0)| \leq |\Hol(L_i)|$ contradicting that $L_0$ is the center leaf with maximal holonomy inside $L^{cs}$.   
\end{proof}
This Lemma directly implies the following: 
\begin{corol}\label{l.maximal}
For every center stable leaf $L^{cs} \in \mathcal{F}^{cs}$ there exists a center leaf $L_0 \in L^{cs}$ 
that intersects  every stable manifold $L^s \subset L^{cs}$ in exactly one point.
\end{corol}
The following Lemma will help us to complete the proof of Proposition~\ref{p.connected}: 
\begin{lemma}\label{l.gamma}
There exists $\gamma > 0$ such that for any center leaf $L \subset L^{cs}$ where $L$ intersects every stable leaf $L^s \subset L^{cs}$ in at most one point the set $$\bigcup_{x \in L}D^s(x,2\gamma) \setminus \bigcup_{x \in L}D^s(x,\gamma)$$ is connected.  
\end{lemma}
We denote $\mathcal{F}^s_{\gamma}(L):=\bigcup_{x \in L}D^s(x,\gamma)$. 
\begin{proof}[Proof of Lemma~\ref{l.gamma}]
A local stable manifold is a disk, hence, there exists $\gamma > 0$ such that for any $x\in M$ the local stable manifold $D^s(x,\gamma)$ is a disk inside $L^s(x)$. As the stable leaf $L^s(x)$ is homeomorphic to $\mathbb{R}^{d_s}$, this implies that $D^s(x,2\gamma) \setminus D^s(x,\gamma)$ is connected. \\
Because $L$ intersects $L^s(x)$ for every $x \in L$ at most once, the set $\mathcal{F}^s_{\gamma}(L)$ is a trivial fiber bundle over $L$, so $\mathcal{F}^s_{\gamma}(L) = D^s(x,\gamma) \times L$, and hence, $\mathcal{F}^s_{2\gamma}(L) \setminus \mathcal{F}^s_{\gamma}(L) = (D^s(x,2\gamma)\setminus D^s(x,\gamma)) \times L$ is a product of connected sets and therefore connected.
\end{proof}
This directly implies that for any center stable leaf $L^{cs}$ and any center leaf $L$ with the properties of Lemma~\ref{l.gamma} $L^{cs} \setminus \mathcal{F}^s_{\gamma}(L)$ is connected. 
\begin{proof}[Proof of Proposition~\ref{p.connected}]
Let $\gamma > 0$ be given by Lemma~\ref{l.gamma}. 
The set $K$ is compact, therefore there exists $n$ such that $f^nK \subset L^s_{\gamma}(f^nL)$. Choose $N:= f^{-n}L^s_{\gamma}(f^nL)$, then Lemma~\ref{l.gamma} implies the statement. 
\end{proof}
Proposition~\ref{p.connected} now implies quite directly Lemma~\ref{l.delta}:
\begin{proof}[Proof of Lemma~\ref{l.delta}]
We take the closure $\overline{U_{m_j}}$ of neighborhoods containing the maximal points $m_j \in L^{cs}(m_1)$ ($m_1$ included) which are finitely many. So the union $\bigcup_j \overline{U_{m_j}}$ is a compact set inside $L^{cs}(m_1)$ which we denote by $\Delta_0$. The set $\Delta_0$ already fulfills the first property of the searched $\Delta$. With Proposition~\ref{p.connected} we assure the existence of a compact set $\Delta \supset \Delta_0$ with the wanted properties which completes the proof. 
\end{proof}


With the arguments above we have established the existence of compact set $\Delta \subset L^{cs}(m_1)$ such that $L^{cs}(m_1) \setminus \Delta$ is connected.\\ 
We denote by $(x,y)^u :=\left\{z \in L^u(x)\;\big|\;x < z < y\right\} \subset L^u_+(x)$ the open arc in the unstable leaf with $x < y$ with respect to the orientation of $\mathcal{F}^u$. 
\begin{lemma}\label{l.emptyintersection}
For all $x \in L^{cs}(m_1) \setminus \Delta$ there exists $y \in L^u_+(x) \cap R$ such that $(x,y)^u \cap R = \emptyset.$ 
\end{lemma}
\begin{proof}[Proof of Lemma~\ref{l.emptyintersection}]
Let $x \in L^{cs}(m_1) \setminus \Delta$ be an arbitrarily chosen point. There exists some neighborhood $V_i$ and a center stable plaque $U_i \times \left\{p\right\}=U_x \subset R$ with $p < 1$ within $V_i$ such that $x$ is contained in it. The point $x$ is not inside $\Delta$, hence, there is no maximal point $m_i$ of intersection inside $U_x \cap R$, so there must exist a point $y \in L^u_+(x) \cap R$ with $y > x$ in sense of orientation of $E^u$. We choose $y$ minimal (with respect to the orientation of $E^u$) with these properties, so it follows $(x,y)^u \cap R$. 
   
\end{proof}
We denote by $I_x=(x,y)^u$ where $y$ is defined by Lemma~\ref{l.emptyintersection} above. 
\begin{lemma}\label{l.help}
Let $x_1, x_2 \in L^{cs}(m_1)\setminus \Delta$ then for every path $\gamma \subset L^{cs}(m_1)\setminus \Delta$ (up to homotopy) with $\gamma(0)=x_1$ and $\gamma(1)=x_2$ there exists a homeomorphism $H_{\gamma}: I_{x_1} \to I_{x_2}$.  
\end{lemma}
\begin{proof}[Proof of Lemma~\ref{l.help}]
Let $x_0,x_1 \in U_i\times\left\{p\right\}$ where $U_i$ is some center stable plaque defined by the cover $(V_i=U_i \times[-1,1])$ of $R$ and $p \in [-1,1]$ and let $\gamma$ be a path from $x_0$ to $x_1$ inside $U_i$. The neighborhood $V_i$ is a product of $U_i$ and $[-1,1]$, so we can lift for every $y \in I_{x_1}$ uniquely the path $\gamma$ to a path $\gamma_y$ starting at $y$. We then define $H_{\gamma}(y)=I_{x_2}\cap \gamma_y$. This map is then a homeomorphism because the path lifting depends continuously on $y \in I_{x_1}$ and the map is injective by the uniqueness of the lifted path and the triviality of the product. \\
As $L^{cs}(m_1)\setminus \Delta$  is covered by finitely many $V_i$ and path connected we can define for any $x_0,x_1$ and path $\gamma(0)=x_0, \gamma(1)=x_1$ inside $L^{cs}(m_1)\setminus \Delta$ the map $H_{\gamma}$ as a concatenation of local homeomorphisms defined on every $V_i$.

\end{proof}
\begin{lemma}\label{l.help2}
Let $x_1 \in L^{cs}(m_1)\setminus \Delta$ and $\gamma$ a closed path with $\gamma(0)=\gamma(1)=x_1$ then $H_{\gamma}=\id$. 
\end{lemma}
\begin{proof}[Proof of Lemma~\ref{l.help2}]
Assume $H_{\gamma}\neq\id$. As $H_{\gamma}$ is an orientation preserving homeomorphism of the interval $I_{x_1}$ there exists a fixed point $p \in \overline{I}_{x_1}$ such that either the interval $(p,p+\epsilon)$ or $(p-\epsilon, p)$ is attracted to $p$. \\
The closed path $\gamma$ can be uniquely lifted to a path $\gamma_p$ through $p$ inside the center stable leaf $L^{cs}(p)$ through $p$. 
Then the center stable leaf $L^{cs}(p)$ has a germ of a holonomy homeomorphism $h_{\gamma_p}$ which is not the identity contradicting that the center stable foliation has trivial holonomy. 
\end{proof}

\begin{lemma}\label{l.help3}
Let $y \in L^u_+(m_1)$ be a point sufficiently close to $m_1$. Then $L^{cs}(y) \subset U$. 
\end{lemma}
\begin{proof}[Proof of Lemma~\ref{l.help3}]
Let $\Delta'\subset L^{cs}(m_1)$ be a compact $\mathcal{F}^c$-saturated set such that $\Delta \subset \interior(\Delta')$. As $\Delta'$ is compact we can cover it with finitely many $V_{i}$ of the cover $(V_i)$. By the choice of $\Delta'$ it can be covered by $\mathcal{F}^c$-saturated stable plaques $U_{k}$ inside $V_{i}$.  Each of these plaques $U_k \cap \Delta'$ can be pushed inside $V_i$ to the $\mathcal{F}^c$-saturated stable plaque through $L^{cs}(y)$, i.e. $V_i \cap L^{cs}(y)$ as $V_i$ is a product neighborhood with respect to the center stable foliation - given $y \in L^u_+(m_1)$ sufficiently close to $m_1$. In this way we obtain a homeomorphic set $\Delta_y'$ through $y$. Fix $y_0 \in \Delta'\setminus \Delta$. By Lemma~\ref{l.emptyintersection} there exists a $y_{R}$ such that $(y_0,y_{R})^u \subset L^u_+(y_0)$ does not intersect $R$. We choose then $y \in L^u_+(m_1)$ in such a way that $\Delta_y'$ intersects $(y_0,y_{R})^u$. Then the homeomorphism $H_{\gamma}$ with respect to a path $\gamma(0)=y_0$ and $\gamma(1)=x$ is well-defined for any $x \in L^{cs}(m_1)\setminus \Delta$. The proof of the following claim ends the proof of the Lemma.\end{proof} 
\begin{claim}
With the notations above, for any $y_0 \in \Delta' \setminus \Delta$ we have 
$$L^{cs}(y) = \Delta'_y \cup \bigcup_{x \in L^{cs}(m_1) \setminus \Delta}H_{\gamma}(\tilde{y}_0)$$
for $\tilde{y}_0=\Delta'_y \cap (y_0,y_{\Omega})^u$ and any path $\gamma$ with $\gamma(0)=y_0$ and $\gamma(1)=x$.
\end{claim}
\begin{proof}
The set is actually equal for any point $y_0 \in \Delta' \setminus \Delta$, so it does not depend on the choice of $y_0$.  
\end{proof}






\section{Trivial Holonomy}\label{s.trivial}
We can now show that the center foliation has only trivial holonomy under the assumption of a one-dimensional oriented unstable direction. First, we prove some general preliminary results about the center holonomy group.

\subsection{Center holonomy}

Let $\phi:D^q \rightarrow M$ be a smooth embedding of a $q$-dimensional disk $D^q$ where $q=\codim \mathcal{F}^c$, $\phi(0)=x$ and $\phi(D^q)=:T$ is transverse at $x$ to $L=L_x\in \mathcal{F}^c$. Then $\mathcal{F}^{cs}$ and $\mathcal{F}^{cu}$ induce foliations on $T$, called $\mathcal{T}^s$ and $\mathcal{T}^u$ respectively. We denote by $H_{\gamma}$ a holonomy homeomorphism generated by $\gamma \in \pi_1(L,x)$, by $H^u_{\gamma}$ a holonomy homeomorphism of the center-stable foliation generated by $\gamma \in \pi_1(L,x)$ and with $H^s_{\gamma}$ a holonomy homeomorphism of the center-unstable foliation. 
\begin{lemma}
Let $f: M \rightarrow M$ be a $C^1$-partially hyperbolic system with an $f$-invariant uniformly compact center foliation $\mathcal{F}^c$. Assume $\dim E^u = 1$. 
For all $\gamma \in \pi_1\left(L, x\right)$, the holonomy map $H_{\gamma}: T \rightarrow T$ is a cartesian product $H^u_{\gamma}\times H^s_{\gamma}$. If $E^u$ is oriented, then it holds $H^u_{\gamma}\equiv \id$. 
\label{lemma_product}
\end{lemma}
\begin{proof}[Proof of Lemma~\ref{lemma_product}]
Let $\gamma \in \pi_1\left(L,x\right)$ be an arbitrary closed path. The path $\gamma$ is tangent to a leaf of $\mathcal{F}^{cs}$ and to a leaf of $\mathcal{F}^{cu}$ as the center foliation $\mathcal{F}^c$ is subordinate to both foliations. As transversal for $H^{u}_{\gamma}$ we can choose $T^u_x$, for $H^{s}_{\gamma}$ in an analogous way $T^s_x$.
As $f$ is dynamically coherent and $\mathcal{T}^u$ and $\mathcal{T}^s$ are induced transversal foliations on $T$, the holonomy $H_{\gamma}|_{T^u_x}$ coincides with $H^{u}_{\gamma}$, in an analogous way the holonomy map $H_{\gamma}|_{T^s_x}$ coincides with $H^{s}_{\gamma}$. Locally, the transverse foliations $\mathcal{T}^s$ and $\mathcal{T}^u$ induce a product structure on $T$. Therefore $H_{\gamma}$ is the cartesian product of $H^{u}_{\gamma} \times H^{s}_{\gamma}$ and preserves the foliations $\mathcal{T}^u$ and $\mathcal{T}^s$. If $E^u$ is oriented, then $\mathcal{F}^{cs}$ is a transversely orientable codimension-1 foliation and every holonomy homeomorphism $H^{u}_{\gamma}:T^u_x \rightarrow T^u_x$ is conjugate to the identity due to Lemma \ref{lemma_1dim}. 
\end{proof}
\begin{rem}
The first statement of Lemma~\ref{lemma_product} that every holonomy map is a cartesian product of the stable and unstable holonomy restricted to the respective center leaf is equally true for all codimensions of a uniformly compact center foliation.
\end{rem}
In the following proofs we consider sets of center leaves with a maximal holonomy group, and it is crucial for the proofs to know that such sets are $f$-invariant. Hence, we need the following quite immediate property of the holonomy group: 
\begin{lemma}
Let $\Hol(L,x)$ be a finite holonomy group of a compact leaf $L \in \mathcal{F}^c$ through $x \in M$. 
Let $H_{\gamma}\in \Hol(L)$ be a periodic holonomy homeomorphism. Then the period of $H_{\gamma}:T \rightarrow T$ is constant along an orbit of $f$, i.e. the period of $H_{f \circ \gamma}$ is equal to the period of $H_{\gamma}$. Every holonomy group $\Hol(f^k(L),f^k(x))$ for $k \in \mathbb{Z}$ is isomorphic to the holonomy group $\Hol(L,x)$ of $L$. 
\label{lemma_period}
\end{lemma}
\begin{proof}[Proof of Lemma~\ref{lemma_period}]
Let $H_{\gamma} \in \Hol(L,x)$ be a periodic holonomy homeomorphism. Let $k\in \mathbb{N}$ be the period of $H_{\gamma}$. The path $f \circ \gamma$ generates a holonomy homeomorphism of $L(f(x))$. Because of the invariance of the foliations we conclude that $H_{f\circ \gamma} = f \circ H_{\gamma} \circ f^{-1}: f(T) \rightarrow f(T)$. Hence, the period of $H_{f\circ \gamma}$ is equally $k$. This induces a bijection between the holonomy homeomorphisms $H_{\gamma} \in \Hol(L,x)$ and $H_{f \circ \gamma} \in \Hol(f(L),f(x))$. Accordingly, the order of the whole holonomy group is constant along a $f$-orbit.  
\end{proof}

Hence, we can say that $f$ is equivariant under $\Hol$, i.e. $f(H(y))= H(f(y))$ for any $H \in \Hol$ and $y \in T$. \\
We define the following set $A$ of all points $x \in M$ whose center leaves have a holonomy group $\Hol\left(L,x\right)$ of maximal order. 
\begin{lemma}
Let $f: M \rightarrow M$ be a partially hyperbolic $C^1$-diffeomorphism with a uniformly compact $f$-invariant center foliation. Then the following hold: 
\begin{enumerate}
\item The order of the center holonomy groups is uniformly bounded. 
\item If $E^u$ is one-dimensional and oriented, then the set $A$ of center leaves having holonomy groups of maximal order is closed, $f$-invariant and $\mathcal{F}^{cu}$-saturated. 
\end{enumerate}
\label{lemma_A}
\end{lemma}
\begin{proof}
The first item is a direct consequence of the Reeb Stability Theorem which gives us the semi-continuity of the map $x \mapsto \left|\Hol(L_x,x)\right|$. So we can show the second item: 
First of all, the $f$-invariance and closedness of $A$ is obvious. It remains to show that $A$ is saturated by unstable leaves:\\
Since $\mathcal{F}^{cs}$ is a transversely orientable codimension-$1$ foliation every holonomy homeomorphism maps $T^u_x$ trivially onto itself. The map $H^{u}_{\gamma}$ coincides with $H_{\gamma}|_{T^u_x}$ and therefore the holonomy group $\Hol(L_y,y)$ restricted to $T^u_x$ for every $y \in T^u_x$ is trivial.
Choose the transversal $T$ at $x$ to $L_x$ such that $T \subset V$ where $V$ is the neighborhood of Theorem \ref{thm1}. Then for every $y \in T^u_x$ it is $\left|\Hol\left(L_y,y\right)\right| \leq \left|\Hol\left(L_x,x\right)\right|$ and $p: L_y  \rightarrow L_x$ is a covering space. For every $y \in V$ the holonomy group $\Hol(L_y,y)$ equals the isotropy subgroup of $\Hol(L_x,x)$ for $y$. As every point $y \in T^u_x$ is fixed by the entire group $\Hol(L_x,x)$, both holonomy groups coincide and $\left|\Hol(L_y,y)\right| = \left|\Hol(L_x,x)\right|$ is implied. \\
As $M$ is compact and the center foliation is uniformly compact, we can find a finite cover of $M$ (in particular of $L^{cu}_x$) ob tubular neighborhoods $p_i: V_i \rightarrow L_i$ as above such that in every neighborhood $V_i$ we have $L^{cu}_{loc}(x) \subset A$ for every $x \in A$. This implies that $L^{cu}_x \subset A$.  

\end{proof}





\subsection{Trivial center holonomy for one-dimensional unstable bundle.}
We can now directly show that the holonomy of every center leaf is trivial if the unstable bundle $E^u$ is one-dimensional and oriented: 
\begin{thm}[Trivial center holonomy]\label{t.trivial}
Let $f: M \rightarrow M$ be a partially hyperbolic $C^1$-diffeomorphism on a compact smooth connected manifold $M$. Assume that the center foliation $\mathcal{F}^c$ is an $f$-invariant uniformly compact foliation and $E^u$ is one-dimensional. Then the following statements hold: 
\begin{enumerate}
\item If $E^u$ is oriented, then the center holonomy is trivial. 
\item If $E^u$ is non-orientable, then center leaves with non-trivial holonomy are isolated, and for every non-trivial holonomy homeomorphism $H$ we have $H^2=\id$. 
\end{enumerate}
\end{thm}
Recall that it is shown in \cite{BB12} that $f$ is dynamically coherent under our current assumptions.
Consider a local transversal $T$ at $x$ to $L_x \in \mathcal{F}^c$ and the foliations $\mathcal{T}^s$ and $\mathcal{T}^u$ induced by $\mathcal{F}^{cs}$ and $\mathcal{F}^{cu}$ respectively on $T$. 
We denote by $A$ as above the set of all center leaves having a holonomy group of maximal order. 
\subsubsection{One-dimensional oriented unstable bundle.}
\begin{lemma}
Let $f: M \rightarrow M$ be a centrally transitive partially hyperbolic $C^1$-diffeomorphism with an $f$-invariant compact center foliation. Assume $\dim E^u =1$ and $E^u$ to be oriented. Then the holonomy of any center leaf is trivial.
\label{lemma_trivial_holonomy}
\end{lemma}
\begin{proof}[Proof of Lemma~\ref{lemma_trivial_holonomy}]
Using the assumption of an oriented unstable bundle, Lemma~\ref{lemma_A} implies that the set $A$ of center leaves having a holonomy group of maximal order is closed, $f$-invariant and $\mathcal{F}^{cu}$-saturated. As $f$ is centrally transitive, there exists a center leaf $L_0$ with a dense forward orbit $\left\{f^n(L_0)\right\}_{n \geq 0}$. Choose any center leaf $L \subset A$. Then there exists $k \geq 0$ such that $f^kL_0$ is sufficiently close to $L$ and the local stable leaf $\mathcal{F}^s_{loc}(f^kL_0)$ intersects the local unstable leaf $\mathcal{F}^u_{loc}(L)$ in a leaf $L_1$ which lies inside $A$ as $A$ is $\mathcal{F}^{cu}$-saturated. Under $f^n$ the iterates $f^nL_1$ and $f^{n+k}L_0$ are asymptotic and since the forward orbit of $f^kL_0$ is dense, so is that of $L_1$. The closedness of $A$ then implies $A=M$. Because the generic center leaf has trivial holonomy (see \cite{EMT77}), we can conclude that all center leaves have trivial holonomy.     
 \end{proof}
\subsubsection{One-dimensional non-orientable unstable bundle.} 

As a direct consequence of Lemma~\ref{lemma_trivial_holonomy} we can show that the order of any center holonomy group is at most $2$: 
\begin{lemma}\label{l.group}
Let $f:M \rightarrow M$ be a partially hyperbolic $C^1$ diffeomorphism having a uniformly compact center foliation. Assume that $E^u$ is one-dimensional and not orientable. Then the group order of any center holonomy group is at most $2$. Further, every non-trivial holonomy homeomorphism reverses the orientation of $E^u$. 
\end{lemma}
\begin{proof}[Proof of Lemma~\ref{l.group}]
Let $p_u: \overline{M}\rightarrow M$ be the $2$-sheeted cover making $E^u$ oriented. Then Lemma~\ref{lemma_trivial_holonomy} implies that the lifted center foliation $p_u^{-1}\mathcal{F}^c$ has only trivial center holonomy. Let $L$ be a center leaf in $M$. Then there are two possibilities for the preimage $p_u^{-1}L$: \begin{enumerate}
\item The lift $p_u^{-1}L$ has two connected components, both diffeomorphic via $p_u$ to $L$. This means that every closed path $\gamma \in \pi_1(L)$ is uniquely mapped to a closed path $\tilde{\gamma} \subset p_u^{-1}L$ by fixing the starting point of $\tilde{\gamma}$. The fundamental group of every connected component $\tilde{L}$ of $p_u^{-1}(L)$ is therefore isomorphic to the fundamental group of $L$, and as the center holonomy of $\tilde{L}$ is trivial, so is the center holonomy of $L$. 
\item The lift $p_u^{-1}(L)$ has one connected component $\tilde{L}$. The map $p_u|_{\tilde{L}}:\tilde{L} \rightarrow L$ is therefore a covering of order $2$; it is the covering which makes $E^u$ along $L$ oriented. Consequently, $\pi_1(\tilde{L})$ is isomorphic to a subgroup of $\pi_1(L)$ of index $2$. The center holonomy group $\Hol(\tilde{L})$ of $\tilde{L}$ is trivial. Consider now the center holonomy group $\Hol(L)$ of $L$ which is isomorphic to $\pi_1(L)/\ker(H)$ where $H$ is the group homomorphism $\pi_1(L) \rightarrow \Homeo(T)$. Let $\gamma \in \ker(H)$ be a closed path, then $E^u$ is oriented along $\gamma$. Therefore it corresponds to a closed path $\tilde{\gamma} \subset \tilde{L}$. On the other hand, consider $p_u\tilde{\gamma}$ for $\tilde{\gamma} \in \pi_1(\tilde{L})$: as the center holonomy of $\tilde{L}$ is trivial, it lies inside the kernel $\ker(H)$. Accordingly, we can identify $p_{u*}(\pi_1(\tilde{L}))$ with $\ker(H)$. Recall that $\Hol(L)=\pi_1(L)/\ker(H)=\pi_1(L)/p_{u*}(\pi_1(\tilde{L}))$. With this we can conclude that the center holonomy group $\Hol(L)$ of $L$ has order $2$ and is generated by a closed path $\gamma$ where $E^u$ is not oriented along $\gamma$, so the holonomy homeomorphism $H_{\gamma}$ generated by such a $\gamma$ reverses the orientation of $E^u$.
\end{enumerate}
\end{proof}  

This result helps us to prove the second item of Theorem~\ref{t.trivial}. Before we start with the proof, we recall that $A$ denotes the $f$-invariant closed set of center leaves having a holonomy group of maximal order which is equivalent thanks to Lemma~\ref{l.group} to having a non-trivial holonomy group. Let $L \subset A$ be a center leaf and $T$ a transversal at $x \in L$ such that the holonomy group $\Hol(L)$ is well-defined in its action on $T$. We state the following easy fact:
\begin{claim}\label{c.setA} 
Let $L$ and $T$ be as above. Then for every nearby center leaf $L' \subset A$ its intersection $L' \cap T$ is a fixed point of the holonomy group $\Hol(L)$. 
\end{claim}
\begin{proof}[Proof of Claim~\ref{c.setA}]
The leaf $L'$ is sufficiently close to $L$, so with the Reeb Stability Theorem the map $p:L' \rightarrow L$ is a covering map. The degree of this covering equals the quotient of $\left|\Hol(L)\right|$ and $\left|\Hol(L')\right|$. As $L'$ and $L$ are both contained in $A$, this quotient is one and the covering map is in fact a diffeomorphism. So $L'$ intersects $T$ in exactly one point which is therefore fixed by $\Hol(L)$. 
\end{proof}  
For simplicity, we say therefore that leaves inside $A$ are fixed points of the holonomy group. 

\begin{lemma}
Let $f: M \rightarrow M$ be a partially hyperbolic $C^1$-diffeomorphism with an $f$-invariant uniformly compact center foliation. Assume $\dim E^u =1$. Then the center leaves with non-trivial holonomy are isolated, and for every non-trivial holonomy homeomorphism $H$ we have $H^2=\id$. 
\label{lemma_isolated_2}
\end{lemma}
\begin{proof}[Proof of Lemma~\ref{lemma_isolated_2}]
Let $p_u: \overline{M} \rightarrow M$ be the orientation cover making $E^u$ orientable. Then the center holonomy for the lifted center foliation is trivial with Lemma~\ref{lemma_trivial_holonomy}. 
Let $L$ be a center leaf with non-trivial holonomy group. Then we know with Lemma~\ref{l.group} that its center holonomy group $\Hol\left(L,x\right)$ has order $2$. Consequently, there exists a closed path $\gamma \in \pi_1(L,x)$ with $\gamma(0)=x$ which generates the non-trivial holonomy homeomorphisms, and we denote such a non-trivial holonomy homeomorphism by $H_{\gamma}=H^s_{\gamma}\times H^u_{\gamma}$. Let $T$ be a transversal to $\mathcal{F}^c$ at $x$, and $T^u_x, T^s_x$ the unstable and stable manifold through $x$ induced by the center stable and center unstable foliation on $T$. We already know that the holonomy homeomorphism $H^u_{\gamma}: T^u_x \rightarrow T^u_x$ reverses the orientation of $E^u$ and has $x$ as its only fixed point. For our proof it suffices to show that $H^s_{\gamma}$ has only isolated fixed points. Assume that $H^s_{\gamma}: T^s_x \rightarrow T^s_x$ has non-isolated fixed points, for instance $x$ is non-isolated. So there exists a sequence $\left\{L_n\right\}_{n \in \mathbb{N}} \subset \mathcal{F}^s_{loc}(L)$ of fixed center leaves converging to $L$. This sequence and all its iterates under $f$ are contained in the set $A$. We will now choose a subsequence of backward iterates which lie each one in a distinct center stable plaque: 
\begin{claim}\label{c.isolated}
There exists a sequence $k_n$ with the following property: If $f^{-k_n}L_n$ and $f^{-i}L_j$ for $0\leq i,j \leq k_{n-1}$ are in the same center stable manifold $\mathcal{F}^s(L_n)$, then their distance within the center stable manifold is greater than $1$. In other words, there is a subsequence of backward iterates of $L_n$ whose elements are contained in an infinite number of distinct center stable plaques.  
\end{claim}
\begin{proof}[Proof of Claim~\ref{c.isolated}]
The leaves $L$ and $L_n \subset \mathcal{F}^s_{loc}(L)$ for $n \geq 0$ are center leaves having a maximal holonomy group, so applying Lemma~\ref{l.maximal} we know that every stable leaf inside $\mathcal{F}^s(L)$ intersects $L$ and every $L_n$ exactly once. So, the stable distance between $L$ and every $L_n$ is well defined by taking the supremum of $d^s(z,L^s(z)\cap L)$ over $z \in L_n$. \\
As this distance is uniformly expanded by $f^{-1}$, we can assure that we find iterates $f^{-k_n}$ of $L_n$ such that they lie in distinct center stable plaques. 
\end{proof}
Now we show that the leaves of this subsequence are still fixed points for the holonomy group of its accumulation point: 
\begin{claim}\label{c.isolated2}
The sequence $f^{-k_n}L_n$ has an accumulation point $L_*$. For $k_n$ sufficiently big, the points $f^{-k_n}L_n$ are fixed points of the center holonomy group of $L_*$.
\end{claim}
\begin{proof}[Proof of Claim~\ref{c.isolated2}]
The leaves $f^{-k_n}L_n$ are all contained in the set $A$ which is compact and $f$-invariant. Therefore, there exists an accumulation point $L_*$ inside $A$. 
By Claim~\ref{c.setA} they are fixed points for the holonomy group of $L_*$. 
\end{proof} 
\begin{claim}\label{c.isolated3}
All fixed points of the holonomy group of $L_*$ lie in its local stable manifold $\mathcal{F}^s_{loc}(L_*)$. 
\end{claim}
\begin{proof}[Proof of Claim~\ref{c.isolated3}]
Let $L$ be a center leaf in a small tubular neighborhood $N$ of $L_*$ and assume that $L$ is fixed by the holonomy group $\Hol(L_*)$ of $L_*$. Let $p_u: \overline{N}\rightarrow N$ be the $2$-sheeted cover of $N$ which makes $E^u$ oriented. Then $L$ lifts to two disjoint copies $L_1,L_2$ in $\overline{N}$, each one diffeomorphic to $L$. Inside $\overline{N}$, $\overline{\mathcal{F}^s_{loc}}(L_i), i=1,2$ intersects $\overline{\mathcal{F}^u_{loc}}(L_*)$ in exactly one center leaf $K_i,i=1,2$ which map each one diffeomorphically to $\mathcal{F}^u_{loc}(L_*) \cap \mathcal{F}^s_{loc}(L)$. Having disjoint preimages under $p_u$, $\mathcal{F}^u_{loc}(L_*) \cap \mathcal{F}^s_{loc}(L)$ has to be $L_*$. Consequently, $L$ lies on the local stable manifold through $L_*$.
\end{proof}
No more than perhaps one center leaf of the sequence $L_j$ is contained inside $\mathcal{F}^s_{loc}(L_*)$. With Claim~\ref{c.isolated3} this implies that they are no fixed points for the holonomy group of $L_*$ contradicting Claim~\ref{c.isolated2}. \\
So we can conclude that the center leaves with non-trivial holonomy are isolated. 
\end{proof}

\subsubsection{Codimension 2}\label{proof_codimension2}
To prove Theorem~\ref{t.trivial} for the case of a compact center foliation with codimension two an alternative way is available which we present here for its own beauty because it utilizes a different idea without recurring to the transitivity of $f$: 

let $f: M \rightarrow M$ be a partially hyperbolic $C^1$-diffeomorphism with a compact center foliation and assume $\dim E^u = \dim E^s =1$. Assume that $E^u$ is oriented. Let $A$ be the set of center leaves having a holonomy group of maximal order. It follows with Lemma~\ref{lemma_A} that $A$ is a $\mathcal{F}^{cu}$-saturated set. The set $A$ is non-empty, so there exists $x \in A$ such that the center unstable manifold $L^{cu}_x \subset A$ which is an unbounded set inside the compact set $A$. Therefore it has to accumulate inside $A$. Let $z \in A$ denote the accumulation point. Then we find a Reeb stability neighborhood $V$ of $z$ and a sequence $\left\{D^{cu}_n(x)\right\}_{n \geq N} \subset V$ of center unstable plaques of $L^{cu}_x$ accumulating at $z$. The center unstable plaques $D^{cu}_n(x)$ intersect the local stable manifold $D^s(z)$ in points $x_n$ converging to $z$ which are therefore fixed points of the center holonomy group $H|_{D^s(z)}$ restricted to the stable direction. As $E^s$ is one-dimensional, this group $H|_{D^s(z)}$ has at most two elements and unless it is the trivial group, $z$ is its unique fixed point. So it is implied that $H|_{D^s(z)}$ is the trivial group, but $z \in A$ which implies that the maximal order of a center holonomy group is one, consequently it is $A = M$ finishing the proof. 

\subsubsection{Codimension 3}\label{proof_codimension3}
This proof is only a sketch to show which kind of arguments might play a role to prove the trivial holonomy: \\
let $f: M \rightarrow M$ be a partially hyperbolic $C^1$-diffeomorphism with a uniformly center foliation and assume $\dim E^u=1, \dim E^s=2$ and $E^u$ is oriented. Define the set $A$ as above. With the same argument as above we can conclude that for every $x \in A$ an arc $\sigma_x \subset L^s(x)$ lies inside $A$ containing the fixed point of the center holonomy group of $x$ restricted to the stable direction. This induces two one-dimensional foliations on $A$, a stable one by $\sigma_x$ and the unstable foliation.\\
Denote with $H$ the holonomy group of maximal order and lift $A$ to $\tilde{A}$ (such that $\tilde{A}/H = A$), the holonomy cover of $H$. Then every connected component of $\tilde{A}/\tilde{\mathcal{F}}^c$ is a two-dimensional tori. As $F$ is transitive, this is impossible.    

\section{Margulis measure}\label{s.margulis}
In this section we prove Theorem~\ref{theorem_measure}. This proof follows the main ideas of the classical proof by Margulis (\cite{M70}) and its modern version in \cite{HK95}. Note that the assumed differentiability of $f$ is now $C^2$ as we have to use the absolute continuity of the stable and unstable foliations. 
\begin{rem}
Theorem~\ref{theorem_measure} under the additional assumption of a one-dimensional unstable direction is more easily proved in the recent article \cite{G11} by Gogolev. He also adapts the proof by Hiraide written down for codimension one Anosov diffeomorphisms for his setting of a partially hyperbolic diffeomorphisms with a compact center foliation with simply connected leaves and a one-dimensional unstable direction - corresponding to our last section~\ref{s.torus}. 
\end{rem}
\begin{rem}[Notations]
\begin{itemize}\item Each leaf of the five $f$-invariant foliations, $\mathcal{F}^u$, $\mathcal{F}^{cu}, \mathcal{F}^c, \mathcal{F}^{cs}, \mathcal{F}^s$, is a Riemannian manifold, so it is endowed with a Riemannian metric and the volume induced by this metric. We denote that volume inside the respective leaf by $\vol$ where the context makes it clear which leaf is meant.  
\item We denote by $D^u(x,\epsilon) \subset L^u$ (or simply by $D^u$) the set of points inside an unstable leaf $L^u$ at unstable distance $< \epsilon$ from $x\in L^u$.
\item We denote by $\mathcal{F}^u(L,\epsilon):=\bigcup_{z \in L} D^u(z,\epsilon)$, the union of local unstable disks through a center leaf $L$. Analogously, $\mathcal{F}^u(L)$ is the union of unstable leaves through the center leaf $L$ which is equal to the center unstable leaf $L^{cu}$ containing $L$.  
\end{itemize} 
\end{rem}
The intuitive idea of the construction is quite simple, and we explain it shortly before we proceed with the formal and quite technical proof:
\subsection{Intuition for proof}
As $f$ is centrally transitive, every center stable and center unstable leaf is dense.
Fix an unstable leaf $L^{u}_0$ of reference and a non-empty open set $A \subset L^{u}_0$ with compact closure inside $L^u_0$. Setting the measure of $A$ to $1$, we would like to define a measure for any open set $B \subset L^{u}$ with compact closure inside $L^u$. 
As any center stable leaf is dense, the center stable leaves through $A$ intersect $B$ (see Lemma~\ref{l.minimal2}). If we regard the $\mathcal{F}^c$-saturates $\mathcal{F}^cA$ and $\mathcal{F}^cB$ of $A$ and $B$, then the positively iterated sets $f^n(\mathcal{F}^cA)$ and $f^n(\mathcal{F}^cB)$ get, in some sense, closer and closer when $n \rightarrow \infty$. The idea is then to define
the measure $\mu_{L^u}(B)$ of the set $B$ 
by  
$$\mu_{L^u} (B) \mbox{``}:=\mbox{''} \lim_{n \rightarrow \infty} \frac{\vol(f^n(\mathcal{F}^cB))}{\vol(f^n(\mathcal{F}^cA))}.$$
However it is not clear that this limit exists. This is the main technical difficulty of the proof.

We will first have to see that the corresponding sequence is bounded.
More precisely, as a consequence of the minimality of the center stable foliation, one can partition $B$ into measurable sets $B_1, \dots B_k$  such that for each $i$ there is a stable holonomy map $h^s$ that sends the saturated set $\mathcal{F}^cB_i$ inside $\mathcal{F}^cA$ (see Corollary~\ref{c.bound}). The absolute continuity of the stable holonomy implies that, for $n \rightarrow \infty$, the quotient of the volume of $f^n(\mathcal{F}^cB_i)$ by the volume of  $f^n(h^s(\mathcal{F}^cB_i))$ converges to $1$  (see Lemma~\ref{l.bound1}). The measure of $f^n(\mathcal{F}^cB)$ will be the sum of the volumes of $\mathcal{F}^cB_i$. So we get that the sequence 
$$\frac{\vol(f^n(\mathcal{F}^cB))}{\vol(f^n(\mathcal{F}^cA))}$$
has an upper bound (and we get a lower bound symmetrically). Some functional analysis will be needed here: the measure $\mu_{L^u}(B)$ of the set $B$ will be established as a limit of a subsequence of convex combinations of elements $\frac{\vol(f^n(\mathcal{F}^cB))}{\vol(f^n(\mathcal{F}^cA))}$. By construction, this limit measure will be invariant under center stable holonomy.  \\
It should be noted that after the adaptation of the definitions \emph{modulo the center foliation} and smaller changes the proof is analogous to the classical proof of the Margulis measure (cf. \cite{M70}) which is nicely explained in \cite{HK95}. So, the main point of this section and the principal idea behind the proof is to define the relevant objects such that we can follow the lines of the classical proof without major changes. The present version of the proof of the Margulis measure seems to allow an easy adaptation to other settings, so that we hope that it might help to discover new and interesting applications for the Margulis measure.   
\subsection{Preliminary results}
Before we start with the proof we state the following facts which are directly implied by our hypothesis. 

The following is  a reformulation of Corollary~\ref{l.maximal} in the trivial holonomy setting:

\begin{lemma}\label{l.uniqueintersection}
Each stable leaf $L^s$ intersects each center leaf $L^c\in \mathcal{F}^c(L^s)$ in exactly one point.
\end{lemma}

This will allow us to find global stable holonomies between unstable neighborhoods of center leaves. 
Let us first define stable holonomies:

\begin{defn}\label{d.stable_holonomy}
Given two $C^1$-submanifolds $V,W\subset M$ transverse to the stable foliation, a {\em stable holonomy map}  $h^s\colon V\to W$ is a continuous map such that $h^s(x)\in L^s(x)\cap W$, for all $x\in V$.
\end{defn}

The {\em stable projection distance} of $h^s$ is the maximum of the distances between $x\in W$ and $h^s(x)$ along the stable foliation:
$$d(h^s)=\max_{x\in W}d^s(x,h^s(x)).$$
We have the following elementary fact on holonomies:

\begin{lemma}\label{l.perturb}
Let $h^s\colon V \to W$ be a stable holonomy map between two manifolds $V$ and $W$ as in Definition~\ref{d.stable_holonomy}, with stable projection distance strictly less than some $\rho>0$. For any manifold $U$ whose closure is included in $V$ and for any manifold $\tilde{U}$ $C^1$-close enough to $U$, there exists a stable holonomy map from $\tilde{U}$ to $W$ with stable projection distance less than $\rho$.
\end{lemma}

The next lemma follows from the existence of a dense orbit of center leaves and the density of periodic center leaves within $M$ - following the classical proof that the stable and unstable foliations of a transitive Anosov diffeomorphism are minimal. 

\begin{lemma}[Minimality of foliations]\label{l.minimal}
Under the assumptions of Theorem~\ref{theorem_measure} the center stable foliation is minimal, i.e. every center stable leaf is dense in $M$. 
\end{lemma}

With the results above we can state the following auxiliary statement:
\begin{lemma}\label{l.minimal2}
Let $A \subset L^u$ be an open non-empty set with compact closure inside the unstable leaf $L^u$. Then there exists $\epsilon > 0$ and $\rho>0$ such that for any $x \in M$ and for any unstable disk $D^u$ of radius $\epsilon > 0$ centered at $x$ there exists a stable holonomy map $h^s$ which sends $\mathcal{F}^cD^u$ inside $\mathcal{F}^cA$, i.e. $h^s(\mathcal{F}^cD^u) \subset \mathcal{F}^cA$, whose stable projection distance is less than $\rho$.
\end{lemma}
\begin{proof}[Proof of Lemma~\ref{l.minimal2}]
Let $L$ be a center leaf going through a point $x\in M$. By Lemma~\ref{l.minimal}, there exists a center leaf $L'\subset \mathcal{F}^cA$ that lies in the same center stable leaf $L^{cs}(x)$. By Lemma~\ref{l.uniqueintersection}, the map
$$h^s: L \rightarrow L', \quad x \mapsto L^s(x)\cap L'$$
is a well-defined bijection. Note that, by transversality of $L$ and $L'$ to the foliation $\mathcal{F}^s$ within $L^{cs}(x)$, that map $h^s$ is a homeomorphism. By the same transversality argument, there is a unique extension of that map to a homeomorphism $h^s\colon \mathcal{F}^u_{loc.}(L)\to \mathcal{F}^u_{loc.}(L')$ from a local unstable neighborhood  of $L$ to a local unstable neighborhood of $L'$, such that $h^s(y)\in L^s(y)\cap \mathcal{F}^u_{loc.}(L')$. As a consequence, for some $\epsilon>0$, the unstable disk $D^u$ of radius $\epsilon > 0$ centered at $x$ satisfies the following: $\mathcal{F}^cD^u$ is sent by a stable holonomy map $h^s$ strictly inside $\mathcal{F}^cA$. By Lemma~\ref{l.perturb}, reducing $\epsilon>0$, we have that any unstable disk $D^u$ of radius $\epsilon > 0$ centered at $y$ close enough to $x$ satisfies the same property, moreover the stable projection distances are all less than some $\rho>0$. We conclude the proof of Lemma~\ref{l.minimal2} by a compactness argument.
\end{proof}

For the existence of the Margulis measure, the absolute continuity of the stable (and unstable) foliation is crucial. The following is a particular case of the classical theorem proved by Brin and Pesin, Pugh and Shub for partially hyperbolic diffeomorphisms: 

\begin{thm}[\cite{BP74},\cite{PS72}]\label{t.acontinuity}
The stable and unstable foliations of the $C^2$ partially hyperbolic diffeomorphism $f$ have absolutely continuous holonomies $h$, i.e. given two manifolds $V,W$ tangent to the center-unstable foliation and a stable holonomy map
$h^s: V \rightarrow W$, then 
$$\vol_{W}(h^s(A)) = \int_A \Jac(h^s(t))d\vol_{V}(t)$$
for every measurable set $A \subset V$. The jacobian is bounded from above and below away from zero. Moreover, when the stable projection distance of $h^s$ goes to $0$, the jacobian $\Jac(h^s)$ converges uniformly to $1$. 
\end{thm}

This Lemma~\ref{l.minimal2} and the absolute continuity of the stable foliation provides us with the following bound which is essential for the existence of the unstable measure:
 \begin{lemma}\label{l.bound1}
Let $A \subset L^u$ be an open non-empty set with compact closure inside the unstable leaf $L^u$. Then, given any non-empty open set $B$ with compact closure inside an unstable leaf such that there exists a stable holonomy map $h^s\colon \mathcal{F}^cB \to \mathcal{F}^cA$, then the  sequence
$$\frac{\vol(f^n(\mathcal{F}^cB))}{\vol(f^n(\mathcal{F}^cA))}$$
where $n\in \mathbb{N}$, is bounded by some constant $C>0$.
\end{lemma} 
\begin{proof}[Proof of Lemma~\ref{l.bound1}]
Given a stable holonomy map $h^s\colon h^s(\mathcal{F}^cB) \to \mathcal{F}^cA$, the map $$h^s_n=f^n\circ h^s\circ f^{-n}\colon f^n(\mathcal{F}^cB)\to f^n(\mathcal{F}^cA)$$ is also a stable holonomy map, and its stable projection distance clearly goes to $0$ as $n\to +\infty$. Then Theorem~\ref{t.acontinuity} implies that $$\vol(h^s_n(f^n\mathcal{F}^cB)) \geq \inf |\Jac(h^s_n)| \vol(f^n\mathcal{F}^cB)$$
and that $\inf |\Jac(h^s_n)|$ converges to $1$. Trivially  $\vol(h^s_n(f^n\mathcal{F}^cB)) \leq \vol(f^n\mathcal{F}^cA)$ and we deduce that the sequence 
$$\frac{\vol(f^n(\mathcal{F}^cB))}{\vol(f^n(\mathcal{F}^cA))}$$
is bounded. \end{proof}

The following corollary is a direct implication of Lemma~\ref{l.minimal2} and guarantees that we can cover every non-empty open set $B$ with compact closure inside an unstable leaf by finitely many sets such that each of their $\mathcal{F}^c$-saturates are sent by stable holonomy maps inside $\mathcal{F}^cA$. This fact allows us to define later the measure for any such $B$.   
\begin{corol}\label{c.bound}
Let $A$ be an open non-empty set with compact closure inside an unstable leaf. For any open non-empty set $B$ with compact closure inside an unstable leaf $L^u$ there exist sets $B_1, \dots, B_k$ inside the unstable leaf $L^u$ such that for every $1\leq i\leq k$ there exists a stable holonomy map $h^s$ which sends $\mathcal{F}^cB_i$ inside $\mathcal{F}^cA$. 
\end{corol}
The following corollary is necessary for the proof of the expansivity of the measure:
\begin{corol}\label{c.bound2}
Let $A$ be as in Lemma~\ref{c.bound} and $\rho>0$ as given by Lemma~\ref{l.minimal2}. For any open non-empty set $B$ with compact closure inside an unstable leaf there exists $n > 0$ such that there exists a stable holonomy map $h^s$ which sends $f^{-n}(\mathcal{F}^cB)$ inside $\mathcal{F}^cA$, with stable projection distance less than $\rho$. 
\end{corol} 
\begin{proof}[Proof of Corollary~\ref{c.bound}]
Lemma~\ref{l.minimal2} implies that there exists $\epsilon > 0$ for $A$ such that any unstable disk $D^u$ of radius $\epsilon$ is mapped into $\mathcal{F}^cA$ by a stable holonomy map of stable projection distance less than $\rho$. For $n$ large enough, $f^{-n}(B)$ is contained inside an unstable disk of radius $\epsilon$.  
\end{proof}
\subsection{Proof of Theorem~\ref{theorem_measure}}
\subsubsection{Construction of the family of measures}
Now we can start defining the principal objects for the proof of Theorem~\ref{theorem_measure}: 
denote by $C_c(L^u)$ the set of continuous functions $\phi: L^u \rightarrow \mathbb{R}$ with compact support. Denote with $C_c^+(L^u)$ the subset of non-negative continuous functions with compact non-empty support endowed with the supremum norm $\left\|\phi\right\|$. Define the sets 
\begin{align*}
C_c&:=\left\{\phi \;\big|\; \exists L^u \in \mathcal{F}^u\;\mbox{such that}\; \phi \in C_c(L^u)\right\}, \\
C_c^+&:=\left\{\phi \;\big|\; \exists L^u \in \mathcal{F}^u\;\mbox{such that}\; \phi \in C_c^+(L^u)\right\}.
\end{align*}
Define the set 
$$\mathfrak{L}:=\left\{G: C_c \rightarrow \mathbb{R}\,\big|\; \text{for all}\; L^u \in \mathcal{F}^u:\; \; G|_{C_c(L^u)}\;\text{is a positive linear functional}.\right\}.$$
The set $\mathfrak{L}$ can be embedded into $\prod_{\phi \in C^+_c} \mathbb{R}_{\phi}$ by the evaluation map $G \mapsto \prod G(\phi)$. We can see the injectivity of this embedding as follows: consider $G,H \in \mathfrak{L}$ which are equal for all non-negative functions $\phi \in C_c^+$, $G(\phi)=H(\phi)$. Every function $\phi \in C_c$ can be written as the difference of two non-negative functions, hence, the linearity of $G,H$ implies that $G(\phi)=H(\phi)$ for all $\phi \in C_c$, so we have $G=H$. \\
The infinite product $\prod_{\phi \in C_c^+} \mathbb{R}_{\phi}$ is a locally convex topological vector space, endowed with the product topology. \\

The next step is the principal difference between Margulis' original construction and our adaptation. 
Symmetrically to Lemma~\ref{l.uniqueintersection}, we have the following:
\begin{lemma}\label{l.uniqueintersection2}
Each stable leaf $L^u$ intersects each center leaf $L^c\in \mathcal{F}^c(L^u)$ in exactly one point.
\end{lemma}
Thanks to that we can extend for every $L^u \in \mathcal{F}^u$ every function $\phi \in C_c(L^u)$ to a function $\tilde{\phi}$ on $\mathcal{F}^{c}(L^u)$ constant along every center leaf. More precisely, we have $\tilde{\phi}(z)=\phi(z)$ for all $z \in L^u$ and $\tilde{\phi}(z)=\tilde{\phi}(z')$ for every $z' \in L^c_{z}$. The new support inside $\mathcal{F}^c(L^u)$ is then the $\mathcal{F}^c$-saturation of $\supp(\phi)$ and therefore also compact. \\
From now on we fix an open non-empty set $A \subset L^u_0$ with compact closure inside $L^u_0$ and a function $\phi_0 \in C_c(L^u_0)$ such that $\supp(\phi_0) \supset A$ and $\phi_0 \geq \chi_A$. \\
With the help of these preliminary definitions and results we can now start to construct the family of unstable measures:\\
define a ``normalized integral with respect to volume'' for all $L^u \in \mathcal{F}^u$
$$H: C_c(L^u) \mapsto \mathbb{R}, H(\phi)=\frac{\int \tilde{\phi} d\vol}{\int \tilde{\phi}_0 d\vol}$$
and the subset 
$$\mathfrak{L}_0:=\left\{G \in \mathfrak{L}\;\big|\; G(\phi_0) = 1\right\}.$$
We have $H \in \mathfrak{L}_0$. 
Now define the following map
$$\Phi: \mathfrak{L}_0 \rightarrow \mathfrak{L}_0, \; \Phi(G)(\phi)=\frac{G(\phi \circ f^{-1})}{G(\phi_0 \circ f^{-1})}\; \text{for all}\; \phi \in C_c.$$
By induction it is easily proved that 
$$\Phi^n(G)(\phi) = \frac{G(\phi \circ f^{-n})}{G(\phi_0 \circ f^{-n})}\quad \text{for all}\; \phi \in C_c, \; n \geq 0.$$
\begin{lemma}\label{l.continuous}
The map $\Phi$ is continuous (with respect to the product topology).
\end{lemma}
\begin{proof}[Proof of Lemma~\ref{l.continuous}]
For $\mathfrak{L}_0 \subset \prod \mathbb{R}_{\phi}$ a basis of the product topology is given by the sets $U \subset \mathfrak{L}_0$ of functionals $G$ such that there exists finitely many functions $\phi_i \in C_c^+, i=1, \dots,m$ and open intervals $I_i, i=1,\dots m$ such that $G(\phi_i) \in I_i$, i.e. $U:=\left\{G \in \mathfrak{L_0}\;\big|\; G(\phi_i) \in I_i\right\}$. So, let $U \subset \mathfrak{L}_0$ be an open set. The set $\Phi^{-1}U:= \left\{G \;\big| \Phi(G) \in U\right\}= \left\{G \;\big| \Phi(G)(\phi_i) \in I_i\right\}$ is then equal to $\left\{G \;\big| G(\phi_i\circ f^{-1}) \in (G(\phi_0 \circ f^{-1}))I_i\right\}$. The functions $\phi_i \circ f^{-1}$ are obviously functions inside $C_c^+$, and $(G(\phi_0 \circ f^{-1}))$ is a constant, so $(G(\phi_0 \circ f^{-1}))I_i$ is still an open interval. Hence, the set\\
$\left\{G \;\big| G(\phi_i\circ f^{-1}) \in (G(\phi_0 \circ f^{-1}))I_i\right\}$ is an open set in the product topology. 
\end{proof}
Consider $\Phi^n(H)$ for $n \geq 0$. 
Then we can show that for all $n \geq 0$ the functionals $\Phi^n(H)$ are contained in a compact subset of $\mathfrak{L}_0$:
\begin{lemma}\label{l.cphi}
For every non-zero $\phi \in C_c^+$ there exist constants $c(\phi),C(\phi) > 0$ such that for all $n \geq 0$ it is  
$$\Phi^n(H) \in \prod_{\phi \in C_c^+} [c(\phi),C(\phi)].$$
\end{lemma}
\begin{proof}[Proof of Lemma~\ref{l.cphi}]
Consider $\phi \in C_c^+(L^u)$ and $\delta > 0$ such that the open set $B:=\left\{z \in L^u_0\;\big|\; \phi_0(z) > \delta\right\}$ is nonempty. The support of $\phi$ is a compact set with non-empty interior inside $L^u$. Applying Lemma~\ref{l.minimal2} to the set $B$ we fix $\epsilon > 0$ associated to $B$. By Corollary~\ref{c.bound} we find finitely many unstable disks $D^u_i, i=1, \dots, k$ of radius $< \epsilon$ inside $L^u$ whose union cover $\supp(\phi)$. Then we get for any $n \geq 0$ that  
$$\int \tilde{\phi} \circ f^{-n}d\vol \leq \vol(\mathcal{F}^c(f^n\supp(\phi))) \left\|\phi\right\| \leq \sum_{i=1}^k \vol(\mathcal{F}^cf^nD^u_i) \left\|\phi\right\|.$$
Applying Lemma~\ref{l.bound1} to $B$ and the sets $D^u_i$ we get for $n \geq 0$ 
$$\int \tilde{\phi} \circ f^{-n}d\vol \leq Ck\vol(f^n(\mathcal{F}^cB))\left\|\phi\right\|.$$
Note that  $\int \tilde{\phi}_0 \circ f^{-n}d\vol\geq \delta\vol(f^n(\mathcal{F}^cB))$, hence
\begin{align*}
\Phi^n(H)(\phi)&=\frac{\int \tilde{\phi} \circ f^{-n}d\vol}{\int \tilde{\phi}_0 \circ f^{-n}d\vol}.\\
&\leq \frac{Ck}{\delta}\left\|\phi\right\|=:C(\phi).
\end{align*}
 Exchanging the roles of $\phi$ and $\phi_0$ we can construct in the same way a lower bound $c(\phi)$.   
\end{proof}
As $\prod_{\phi}[c(\phi),C(\phi)]$ is a convex compact set, for all $k \geq 0$ the convex hull $\co_{n \geq k}(\Phi^n(H))$ is also contained inside $\prod_{\phi}[c(\phi),C(\phi)]$ and its closure $C_k := \overline{\co_{n \geq k}(\Phi^n(H))}$ is a convex compact set. \\
This implies that the intersection $C_{\infty}:= \bigcap_{k \geq 0} C_k$ is also a convex compact set. 
\begin{lemma}\label{l.invariance}
For every $k \geq 0$ we have $\Phi(C_k) \subset C_k$.
\end{lemma}
\begin{proof}[Proof of Lemma~\ref{l.invariance}]
Consider $G= \sum_{i=0}^m c_i \Phi^{n_i}(H)$ with $n_i \geq k$ and $\sum_{i=0}^m c_i = 1$. The element $G$ lies inside $\co_{n \geq k}(\Phi^n(H))$. We will show by direct calculation that $\Phi(G)$ lies inside $\co_{n \geq k+1}(\Phi^n(H))$. Let $\phi \in C_c$.
\begin{align*}
\Phi(G(\phi))&= \Phi(\sum_{i=0}^m c_i \Phi^{n_i}(H))(\phi)= \frac{\sum_{i=0}^m c_i \Phi^{n_i}(H)(\phi \circ f^{-1})}{\sum_{i=0}^m c_i \Phi^{n_i}(H)(\phi_0 \circ f^{-1})}\\
&=\frac{\sum_{i=0}^m c_i \Phi^{n_i+1}(H)(\phi)}{\sum_{i=0}^m c_i \Phi^{n_i+1}(H)(\phi_0)}.\\
\end{align*}
So $\Phi(G)$ is a finite convex combination with $n_i+1 \geq k+1$ for $i=0,\dots,m$ and $$\frac{\sum_{i=0}^mc_i}{\sum_{i=0}^m c_i \Phi^{n_i+1}(H)(\phi_0)} =1$$
because $\sum_{i=0}^m c_i \Phi^{n_i+1}(H)(\phi_0)=1$ by definition of $H$. Thus, $\Phi(G) \in \co_{n \geq k+1}(\Phi^n(H))$. Utilizing the continuity of $\Phi$ we can conclude.  
\end{proof}
\begin{lemma}\label{l.holonomyinvariant}
Every element $G \in C_{\infty}$ is holonomy invariant, that is, for any pair $\phi,\psi \in C^+_c$ such that $\tilde{\phi} = \tilde{\psi} \circ h^s$, for some stable holonomy map $h^s$, we have $G(\phi)=G(\psi)$.
\end{lemma}
\begin{proof}[Proof of Lemma~\ref{l.holonomyinvariant}]
Let $G \in C_{\infty}$. Then for every neighborhood $U$ of $G$ there exists $N \geq 0$ such that $\sum_{i=0}^mc_i\Phi^{n_i}(H)$ with $n_i \geq N$ lies inside $U$. Let $\phi,\psi \in C^+_c$ such that $\tilde{\phi} = \tilde{\psi} \circ h^s$. Then we prove the following claim:
\begin{claim}\label{c.holonomyinvariant}
For every $\epsilon > 0$ there exists $N \geq 0$ such that 
$$\int \tilde{\phi} \circ f^{-n} d\vol - \int \tilde{\psi} \circ f^{-n}d\vol \leq \epsilon \int \tilde{\phi}_0 \circ f^{-n} d\vol$$
for $n \geq N$. 
\end{claim}
\begin{proof}[Proof of Claim~\ref{c.holonomyinvariant}]
The map $h^s_n=f^n\circ h^s\circ f^{-n}$ is also a holonomy map and $\tilde{\phi}\circ f^{-n}=\tilde{\psi}\circ f^{-n}\circ h^s_n$.
The claim is a direct implication of the absolute continuity and Lemma~\ref{l.bound1}: 
\begin{align*}
&\frac{\int \tilde{\phi} \circ f^{-n} d\vol - \int \tilde{\psi} \circ f^{-n}d\vol}{ \int \tilde{\phi}_0 \circ f^{-n} d\vol}\\
&\leq \sup|\Jac(h^s_n) - 1| \frac{\int \tilde{\psi} \circ f^{-n} d\vol}{ \int \tilde{\phi}_0 \circ f^{-n} d\vol}.
\end{align*}
The fraction $\frac{\int \tilde{\psi} \circ f^{-n} d\vol}{ \int \tilde{\phi}_0 \circ f^{-n} d\vol}$ for $n \geq 0$ is bounded because the functions $\psi$ and $\phi_0$ are bounded and the fraction for their supports is bounded by $C>1$ associated to $A$ by Lemma~\ref{l.bound1}. The jacobian of $h^s_n$ tends to $1$ with $n \rightarrow \infty$, as seen in the proof of Lemma~\ref{l.bound1}, therefore there exists $N \geq 0$ such that the claim is proved.   
\end{proof}  
Hence, for every $\epsilon > 0$ we can find a neighborhood $U$ of $G$ such that $G'(\phi)$, $G'(\psi)$ lie in small open intervals for $G' \in U$ and such that $\sum_{i=0}^mc_i\Phi^{n_i}(H)$ with $n_i \geq N$ lies inside $U$ and with the Claim above we get: 
$$\left|\sum_{i=0}^mc_i\Phi^{n_i}(H)(\phi) - \sum_{i=0}^mc_i\Phi^{n_i}(H)(\psi)\right| < \epsilon \, \quad \text{for}\; n_i \geq N.$$
So we can conclude that $G(\phi)=G(\psi)$, and $G$ is therefore holonomy invariant. 
\end{proof}
The following fixed point theorem holds in our context: 
\begin{thm}[Tychonov, Schauder]
Let $X$ be a locally convex topological vector space and $\emptyset \neq E \subset 
F \subset X$ with compact $E$ and convex $F$. Then every continuous $f: F \rightarrow E$ has a fixed point. 
\end{thm}
Applying this theorem to $\Phi: C_{\infty} \rightarrow C_{\infty}$ we can conclude that there exists $L \in C_{\infty}$ such that $\Phi(L)=L$. 
By the Riesz representation theorem (cp Theorem 2.14 in \cite{R87}) there exists an isomorphism between the space of positive linear functionals on $C_c(L^u)$ and the space of positive regular Borel measures on $L^u$. Therefore we can map $L|_{C_c(L^u)}$ for every $L^u \in \mathcal{F}^u$ to a positive regular Borel measure $\mu_{L^u}$ on $L^u$. The family $\left\{\mu_{L^u}\right\}_{L^u \in \mathcal{F}^u}$ is then the family of measures we searched for.
\subsubsection{Properties of the family of measures}
To complete the proof of Theorem~\ref{theorem_measure} we still have to prove the last properties, namely, that $\left\{\mu_{L^u}\right\}$ is expanding and every measure $\mu_{L^u}$ non-atomic. For this purpose define $\theta > 0$ for $\phi_0$ such that $\phi(x) > \theta$ on an open set and denote this set by $B_0$. Further, denote the forward iterates $f^k(B_0)$ by $B_k$. For the proof of the expansivity we need then the following proposition: 
\begin{prop}\label{c.principal}
For every $n \geq 0$ and $\delta > 0$ there exists $k \geq 0$, disjoint open sets $U_i \subset B_k$ and holonomy maps $h^s_i: \mathcal{F}^cU_i \rightarrow \mathcal{F}^cL^u_0$ whose stable projection distance is less than $\delta$ and $\left\{h^s_i(\mathcal{F}^cU_i)\right\}$ covers the saturated support $\supp(\tilde{\phi}_0) \subset \mathcal{F}^cL^u_0$. Further, for all center leaves $L \subset \supp(\tilde{\phi}_0)$ there exist $n$ disjoint center leaves $L_1 \in \mathcal{F}^cU_{j_1}, \dots, L_n \in \mathcal{F}^cU_{j_n}$ such that $h^s_{j_i}(L_i)=L$ for $i=1,\dots,n$. 
\end{prop}
\begin{proof}[Proof of Proposition~\ref{c.principal}]
Fix an arbitrary $n \geq 0$ and $\delta > 0$. Choose $n$ disjoint unstable disks $D^u_i \subset B_1$ with $i=1, \dots, n$.
Applying Corollary~\ref{c.bound2} to the support $\supp(\phi_0)$ and every unstable disk $D^u_i$ we find $\rho>0$ such that for large $N \geq 0$, there are stable holonomy maps $h^s_i$ with stable projection distance less than $\rho$, such that $f^{-N}\supp(\phi_0)$ is contained in a sufficiently small unstable disk $D^u_0 \subset f^{-N}L^u_0$ and consequently sent inside $\mathcal{F}^cD^u_i$, i.e. $h^s_i(f^{-N}\supp(\tilde{\phi}_0)) \subset \mathcal{F}^cD^u_i$. So we get that the support $\supp(\tilde{\phi}_0)$ is sent by a holonomy map $h^s_i$ inside $f^N(\mathcal{F}^cD^u_i)$ for $i=1, \dots, n$ which are disjoint sets inside $\mathcal{F}^cB_{N+1}$. There are constants $C,\lambda>0$ such that the stable projection distance between $\supp(\tilde{\phi}_0)$ and $f^N(\mathcal{F}^cD^u_i)$ by the holonomy map $f^N\circ h^s_i\circ f^{-N}$ is less than $Ce^{-\lambda N}\rho$. By choosing $N$ sufficiently big, we get a stable projection distance less than $\delta$. Define the images $h^s_i(\mathcal{F}^cD^u)$ as $\mathcal{F}^cU_i$ lying inside $\mathcal{F}^cB_{N+1}$.       
Further, we have that $(h^s_i)^{-1}(\mathcal{F}^cU_i)$ cover $\supp(\tilde{\phi}_0)$ and therefore $\mathcal{F}^cB_0$. Every center leaf $L \subset \mathcal{F}^cB_0$ is covered by $n$ sets $\mathcal{F}^cU_i$ and has therefore $n$ disjoint preimages inside $\mathcal{F}^cB_{N+1}$. By chosing $k=N+1$ and $U_1, \dots U_n$ as constructed above we prove the Lemma.   
\end{proof}
\begin{lemma}\label{l.expanding2}
There exists $\alpha > 1$ and $k \in \mathbb{N}$ such that for all $n \geq k$ we have $\Phi^n(H)(\phi_0 \circ f^{-k}) > \alpha.$
\end{lemma}
\begin{proof}[Proof of Lemma~\ref{l.expanding2}]
Recall that $\supp(\phi_0) \subset L^u_0$. Define\\ $B_0:=\left\{x \in L^u_0\;|\; (\phi_0 \circ f^{-n})(x) > \theta\right\}$ and $B_k:=f^kB_0$ for $k \geq 0$ as above. 
We fix $N > \frac{2 \max \phi_0}{\theta}$. Proposition~\ref{c.principal} implies the existence of $k$ and by choosing $\delta >0$ small enough we can guarantee that $\Jac(h^s_i) > 1/2$ for $i=1, \dots,N$. This implies Lemma~\ref{l.expanding2} as follows:
\begin{align*}
\int \tilde{\phi}_0 \circ f^{-n-k} d\vol|_{L^{cu}} &\geq \epsilon \vol(\mathcal{F}^cB_{n+k})\\
&\geq \epsilon \sum_{i=1}^N \vol(\mathcal{F}^cU_i)\\
& > \frac{\epsilon}{2} \sum_{i=1}^N\vol(h^s_i(\mathcal{F}^cU_i))\\
& \geq \frac{\epsilon}{2}N \vol(\supp(\tilde{\phi}_0)).
\end{align*}
Dividing by $\int \tilde{\phi}_0 \circ f^{-k}d\vol$ we get
$$\frac{\int \tilde{\phi}_0 \circ f^{-n-k} d\vol}{\int \tilde{\phi}_0\circ f^{-n} d\vol} \geq \frac{\epsilon N \vol(\supp(\tilde{\phi}_0 \circ f^{-n}))}{2\left\|\phi_0\right\|\vol(\supp(\tilde{\phi}_0\circ f^{-n}))}.$$
The choice of $N$ then directly finishes the proof.
\end{proof} 

Now we can quite directly deduce that the measure family is expanding under $f$. 
\begin{lemma}\label{l.expanding}
There exists $\beta > 1$ such that $\mu_{L^u}(f(C))= \beta\mu_{L^u}(C)$ for every non-empty open set $C \subset L^u$ and every $L^u \in \mathcal{F}^u$. 
\end{lemma}
\begin{proof}[Proof of Lemma~\ref{l.expanding}]
The existence of a constant $\beta$ is clear: the functional $L$ is a fixed point of $\Phi$, therefore we have 
$$\Phi(L)(\phi)=\frac{L(\phi \circ f^{-1})}{L(\phi_0 \circ f^{-1})} = L(\phi)\, \quad\text{for all}\; \phi \in C_c, $$ and naturally also $\Phi^kL(\phi)=L(\phi)$ for $k \geq 0$. 
So, it is $\beta:=L(\phi_0 \circ f^{-1})$ and $\beta^k=L(\phi_0 \circ f^{-k})$. Therefore we only need to prove that $\beta > 1$.\\
For that it suffices to show that there exists $k \in \mathbb{N}$ such that $L(\phi_0 \circ f^{-k}) > 1$. As $L \in C_{\infty}$, for every neighborhood $U$ of $L$ there exists $n$ such that the element $\Phi^n(H) \in U$. Accordingly, Lemma~\ref{l.expanding2} straightforwardly implies Lemma~\ref{l.expanding}.
\end{proof}

Now we can finally prove the last property which is an easy consequence of the minimality of the center stable foliation and the holonomy invariance of $\left\{\mu_{L^u}\right\}$: 
\begin{corol}\label{c.nonatomic}
For every unstable leaf $L^u \in \mathcal{F}^u$ the measure $\mu_{L^u}$ constructed above is non-atomic. 
\end{corol}
\begin{proof}
Assume there exists $L^u \in \mathcal{F}^u$ and $y \in L^u$ such that $\mu_{L^u}(\left\{y\right\}) > 0$. The family of measures is expanding under $f$, therefore, every point $f^n(y)$ for $n \geq 0$ is also an atom with a positive measure $\beta^n \mu_{L^u}(\left\{y\right\})$. Consequently, there exists $N$ such that $\mu_{f^NL^u}(\left\{f^N(y)\right\}) >1.$ \\
Fix $\epsilon > 0$ associated to our set $A \subset L^u_0$ given by Lemma~\ref{l.minimal2} and consider $D^u(f^N(y),\epsilon)$, then there exist a subset $C_N$ of $A$ and a holonomy map $h^s$ such that $\mathcal{F}^cC_N= h^s_N(\mathcal{F}^cD^u(f^N(y),\epsilon))$ and their measures are consequently equal: \\
$\mu_{L^u_0}(C_N)=\mu_{f^NL^u}(D^u(f^N(y),\epsilon))$. It is clear that $C_N$ as a subset of $A$ and so $\mu_{f^NL^u}(D^u(f^N(y),\epsilon))$ have a measure smaller or equal $1$. Accordingly, we get a contradiction by the fact that $\mu_{f^NL^u}(D^u(f^N(y),\epsilon)) > d^N \mu_{f^NL^u}(\left\{f^N(y)\right\}) > 1 $
\end{proof}
 We have now established a holonomy invariant family of non-atomic Borel measures $\mu_{L^u}$ which are expanding under $f$, so the proof of Theorem~\ref{theorem_measure} is finished. \\
\subsubsection{A final property}
We can also show the following property of our family of measures $\left\{\mu_{L^u}\right\}$ which will be useful in the following Section~\ref{s.torus}:
\begin{corol}\label{c.unbounded}
The measure $\mu_{L^u}(D^u(x,R))$ of any unstable disk of radius $R > 0$ tends to infinity for $R \rightarrow \infty$. 
\end{corol}
The corollary is quite directly implied by the following slight generalization of Lemma~\ref{l.minimal2}:
\begin{lemma}\label{l.minimal3}
For any $R > 0$ there exist $\epsilon > 0$ and $\rho > 0$ such that for any disk $D^u(R)$ of radius $R$ and any disk $D^u(\epsilon)$ of radius $\epsilon$ there exists a stable holonomy map $h^s$ which maps $\mathcal{F}^c\left(D^u(\epsilon)\right)$ into $\mathcal{F}^c\left(D^u(R)\right)$ such that the stable projection distance is strictly less than $\rho > 0$.
\end{lemma}
\begin{proof}[Proof of Lemma~\ref{l.minimal3}]
Fix an arbitrarily chosen $R >0$. 
By Lemma~\ref{l.minimal2} we have the following: for any unstable disk $D^u(R)$ there exists $\epsilon > 0$ and $\rho > 0$ with the properties above. We have to prove that $\epsilon, \rho > 0$ can be chosen independently of the set $D^u(R)$, depending only on the radius $R > 0$. But this is just a consequence of the compactness of the manifold $M$.   
\end{proof}
Now we can prove the Corollary in the following way:
\begin{proof}[Proof of Corollary~\ref{c.unbounded}]
Choose $R > 0$. We fix $\epsilon(R) > 0$ given by Lemma~\ref{l.minimal3} and an unstable disk $D^u(\epsilon)\subset L^u_0$ which has certainly positive measure $\mu_{L^u_0}(D^u(\epsilon)) > 0$. Then consider an arbitrarily chosen unstable disk $D^u(x,NR)$ centered at $x$ of radius $NR$ with $N\in \mathbb{N}$. Then $D^u(x,NR)$ contains at least $N$ disjoint unstable disks $D^u(x_i,R), i=1,\dots,N$ of radius $R$ centered at points $x_i \in D^u(x,NR)$. By Lemma~\ref{l.minimal3} the disk $D^u(\epsilon)$ is mapped into $D^u(x_i,R)$ by stable holonomy, and the holonomy-invariance of our measure implies 
$$\mu_{L^u}(D^u(x,NR)) > \sum_{i=1}^N\mu_{L^u}(D^u(x_i,R)) \geq N*\mu_{L^u_0}(D^u(\epsilon)).$$ 
Accordingly, if $N$ goes to infinity, the measure $\mu_{L^u}(D^u(x,NR))$ also goes to infinity. 
\end{proof}

\subsection{Margulis measure}
For the classification of partially hyperbolic diffeomorphisms with a one-dimensional unstable direction we do not need a complete Margulis measure. But as we have already constructed the family of unstable measures it is quite easy to complete the construction obtaining the complete Margulis measure. Accordingly, we give here a sketch of the proof.
A family of measures $\left\{\mu_{L^s}\right\}$ on the stable leaves are constructed in the exact symmetric way as the family of unstable measures $\left\{\mu_{L^u}\right\}$ above exchanging $f^{-1}$ by $f$. This family is by construction invariant under the center unstable holonomy $h^u$ and contracted by $f$ by a constant $0< \gamma<1$. It is left to prove that $\gamma=\beta^{-1}$.
With these two family of measures we can construct a complete $f$-invariant measure $\mu$ on $M$ by taking the local product of the measures: 
let $U_i$ be a $\mathcal{F}^c$-saturated product neighborhood as defined at the beginning and define $\mu_{U_i}$ on $U_i$ as follows for any non-empty open $\mathcal{F}^c$-saturated set $A\subset M$ with compact closure: 
$$\mu_{U_i}(A)=\int_{\mathcal{F}^s_{loc}(L)}\mu_{L^u}(A \cap \mathcal{F}^u_{loc}(L'))d\mu_{L^s}(L').$$
To assure the well definedness of this measure we have to show that $L' \mapsto \mu_{L^u}(A \cap \mathcal{F}^u_{loc}(L'))$ is a measurable map. This local measure can then be canonically extended to the whole $M$.    
\begin{rem}\begin{enumerate}
\item Looking at the ingredients for the construction of the family $\left\{\mu_{L^u}\right\}$ you recognize that there are three principal: given two foliations $\mathcal{F}, \mathcal{G}$ transverse to each other it is necessary for a family of measures on leaves of $\mathcal{F}$ that the local holonomy along leaves of $\mathcal{G}$ is absolutely continuous. Further, $\mathcal{G}$ must be a minimal foliation. So, if the stable foliation is minimal for a partially hyperbolic $C^2$ diffeomorphism with a compact center with trivial holonomy, then we can in an analogous way construct a family of measures on the center unstable leaves invariant by the strong stable holonomy. You can think of many other dynamics which fulfill these three assumptions allowing an adaptation of the proof above.   
\item While certainly quite technical because of the use of local holonomy covers we think that it is nevertheless possible to adapt the proof for the finite holonomy setting.  
\end{enumerate}
\end{rem}
\subsection{Induced measure on the quotient space.}
Under the hypothesis of this chapter, the quotient space $\pi: M \rightarrow M/\mathcal{F}^c$ of the compact center foliation with trivial holonomy is a topological manifold. 
The family $\left\{\mu_{L^u}\right\}$ of measures can be seen as a family of measures $\left\{\mu_{L^u}\right\}$ on the quotient space such that $\mu_{L^u}$ has its support on the unstable leaf $\pi L^{cu}$: 
every measure $\mu_{L^u}$ is defined on an unstable leaf of the unstable foliation in $M/\mathcal{F}^c$ induced by the center unstable foliation $\mathcal{F}^{cu}$ via $\pi$. 
In the same way, we obtain a measure family $\mu_{L^s}$ on the stable leaves obtained by projecting the stable measure family onto $M/\mathcal{F}^c$. As the complete measure $\mu$ is defined for $\mathcal{F}^c$-saturated sets we can directly project it to a measure on the quotient space $M/\mathcal{F}^c$.  
\begin{rem} Oxtoby and Ulam showed in \cite{OU41} that any Borel measure on $I^n$, the $n$-dimensional cube, is topologically equivalent to the $n$-dimensional Lebesgue-Borel measure if and only if it is everywhere positive (that is, positive for non-empty open sets), non-atomic, normalized and vanishes on the boundary. Hence, especially in the case that the unstable bundle is one-dimensional, then every measure $\mu_{L^u}$ restricted to an open interval is topologically equivalent to Lebesgue $\lambda$, that is, for every $x$ there exists a homeomorphism $h$ such that $h_*(\lambda)=\mu_{L^u}$.  
\end{rem}
 \begin{conj}
 Under the hypothesis of Theorem~\ref{theorem_measure} it is not difficult to see that the induced $F$-invariant measure $\mu$ on the leaf space $M/\mathcal{F}^c$ is an ergodic measure for the induced quotient dynamics $F$. We conjecture that it is the entropy maximizing measure for $F$ being constructed as Margulis measure. 
 \end{conj}

\section{Quotient dynamics is conjugate to hyperbolic torus automorphism }\label{s.torus}
In this section we follow more or less the proof by Hiraide (\cite{H01}) of Franks' Theorem (\cite{F70}). So, we refer the interested reader for details to the original proof and enlist here only the main steps ensuring that the lack of differentiability on the quotient space does not create any problems for us. We also would like to cite \cite{G11} who adapts the same proof to his setting. \\
Franks' Theorem states that any transitive Anosov diffeomorphism with a one-dimensional unstable (or stable) direction is conjugate to a hyperbolic torus automorphism. We establish this statement for the induced dynamics on the quotient space of a compact center foliation with trivial holonomy where the induced unstable foliation is one-dimensional. The quotient space is a priori a topological manifold, and the induced dynamics is a homeomorphism. \\  
Hence, we show the following theorem:  
\begin{otherthm}
Let $f: M \rightarrow M$ be a partially hyperbolic $C^{\infty}$-diffeomorphism with an $f$-invariant uniformly compact center foliation. Assume that $E^u$ is one-dimensional and oriented. Then $F: M/\mathcal{F}^c \rightarrow M/\mathcal{F}^c$ is topologically conjugate to a hyperbolic toral automorphism and $M/\mathcal{F}^c$ is homeomorphic to a $\codim \mathcal{F}^c$-torus.  
\label{theorem_franks}
\end{otherthm}
Theorem~\ref{theorem_franks} implies our main Theorem~\ref{maintheorem} in the following way: let $f$ be as in Theorem~\ref{theorem_franks} but only $C^1$. Using the result of the structural stability of $F$ proved in \cite{BB12} we can perturbe $f$ to obtain a partially hyperbolic $C^{\infty}$-diffeomorphism $g$. Its center foliation $\mathcal{F}^c_g$ is then leafwise conjugate to the center foliation $\mathcal{F}^c_f$ and so are $F$ and $G$. The induced dynamics $G$ is by Theorem~\ref{theorem_franks} conjugate to a hyperbolic toral autormophism, and so is the induced dynamics $F$ which proves Theorem~\ref{maintheorem}.   
\begin{rem}
We know by Theorems~\ref{thm_transitive} and \ref{t.trivial} that the center foliation in the setting of Theorem~\ref{theorem_franks} has only trivial holonomy and that therefore the leaf space $M/\mathcal{F}^c$ is a compact topological manifold. The center-unstable and center-stable foliations in $M$ projects to transverse stable and unstable topological foliations in $M/\mathcal{F}^c$. We recall that the map $F$ is shown to be expansive and transitive on the leaf space $M/\mathcal{F}^c$, and it has the (unique) pseudo-orbit tracing property. \\
The projected measure family $\left\{\mu_{L^u}\right\}$ introduced in the section above consists of non-atomic measures $\mu_{L^u}$ defined on the one-dimensional unstable leaves, positive on non-empty open sets, hence, they are all topologically equivalent to the Lebesgue measure. 
\end{rem}
For the definition of the covering map we need the constructed unstable measure $\mu_{L^u}$ to be infinite if and only if an interval inside $L^{u}$ is unbounded. This is a direct consequence of the construction of the measure and of Corollary~\ref{c.unbounded}: 
\begin{corol}
For any $L^u\in \mathcal{F}^u$ we have $\mu_{L^u}(I)=\infty$ if and only if $I \subset L^{u}$ is an unbounded interval.
\label{corol_measure_infinity}
\end{corol}
Hiraide's proof is divided into four main steps: 
\begin{enumerate}
\item Every stable manifold $L^s$ is homeomorphic to $\mathbb{R}^s$, $s=\dim \mathcal{F}^s$, and every unstable manifold $L^u$ to $\mathbb{R}$. Assuming that $p \in M/\mathcal{F}^c$ is a fixed point of $F$, we show that $L^s_p\times L^u_p = \mathbb{R}^{s+1}$ is a universal cover of $M/\mathcal{F}^c$.
\item The fundamental group $\pi_1(M/\mathcal{F}^c)$ is isomorphic to $\mathbb{Z}^m$ for some $m \in \mathbb{N}$. 
\item The lifted map $\tilde{F}$ of $F$ on the universal cover $L^s_p \times L^u_p$ is topologically conjugate to a linear hyperbolic map $A \in \Gl(m,\mathbb{Z})$. 
\item Finally, we can conclude that this topological conjugacy descends to a topological conjugacy between $F$ and the hyperbolic toral automorphism $\phi_A$ induced by $A$ establishing a homeomorphism between $M/\mathcal{F}^c$ and $\mathbb{T}^m=\mathbb{R}^m/\mathbb{Z}^m$ as well (what implies that $m = s+1$).   
\end{enumerate}

We only give the proof of the first step (Lemma~\ref{lemma_universal_cover} below) as the remaining coincides with Hiraide's proof in \cite{H01}. Before stating the Lemma, we need some formalism: we have a uniform local product structure on $M/\mathcal{F}^c$ given by the induced transverse stable and unstable foliations, so there exists a cover of product neighborhoods $U$ such that we have well defined local holonomy homeomorphism $h^u: I \rightarrow J$ for any $I\subset L^u_1, J \subset L^u_2$ with $I,J \subset U$ defined by $z \in I \mapsto L^s_{loc}(z) \cap J$. We call any finite composition of local holonomy homeomorphisms a holonomy homeomorphism. \\
We may assume that $F:M/\mathcal{F}^c\rightarrow M/\mathcal{F}^c$ has a fixed point $p$: in any case as a consequence of the Shadowing Lemma it has a periodic point $p$ such that $F^n(p)=p$, so we could prove the statement for $G:=F^n$. Then $F^n$ is conjugate to a hyperbolic toral automorphism, and $F$ is therefore homotopic to a hyperbolic toral automorphism. As the fixed point index is a homotopy invariant, we can conclude that $F$ has a fixed point. \\
Therefore, let $p \in M/\mathcal{F}^c$ be a fixed point for $F$. We have $\dim \mathcal{F}^u=1$ and - as $E^{u}$ in $M$ is oriented - we may fix a direction on arcs of $\mathcal{F}^u$. We will construct the simply connected cover for $M/\mathcal{F}^c$ as $L^u_p \times L^s_p$ which is homeomorphic to $\mathbb{R}^{q}$ with $q = \codim \mathcal{F}^c$. Let $(x,y)\in L^u_p \times L^s_p$. For any $x \in L^u_p$ we denote by $\left[p,x\right]$ the arc inside $L^u_p$ with end points $p$ and $x$ and for any $z \in L^u_y$ by $\left[y,z\right]$ the arc inside $L^u_y$ with end points $y$ and $z$. 
\begin{lemma}
The manifold $M/\mathcal{F}^c$ has a universal cover $$\pi_p: L^u_p \times L^s_p \rightarrow M/\mathcal{F}^c$$ where $\pi_p(x,y)= z$ such that $\mu_{L^u}(\left[p,x\right])=\mu_{L^u}(\left[y,z\right])$ and the orientation from $p$ to $x$ is consistent with the orientation from $y$ to $z$. Further, $L^u_p \times L^s_p$ is homeomorphic to $\mathbb{R}^q$ where $q = \codim \mathcal{F}^c$. 
\label{lemma_universal_cover}
\end{lemma}
\begin{proof}[Proof of Lemma~\ref{lemma_universal_cover}]
The most important step in the construction of the universal cover is the proof that the map $\pi_p$ is well-defined. This is done with the help of the family of measures $\mu_{L^u}$ on the unstable leaves. The family of measures is invariant under the holonomy homeomorphisms $h^u$. By the minimality of the foliations it is then implied that for any $(x,y) \in L^u_p \times L^s_p$ that we can map via holonomy a local unstable segment around $x$ onto a local unstable segment around $y$. Therefore there exists a point $z \in L^u_y$ such that $$\mu_{L^u}\left(\left[p,x\right]\right) = \mu_{L^u}\left(\left[y,z\right]\right)$$ and the orientation from $p$ to $x$ is consistent with the orientation from $y$ to $z$. This point is unique as the measure $\mu_{L^u}$ is non-atomic and positive on any non-empty open interval, accordingly, if there would be two points $z_1 < z_2$ in $L^u_y$ it would follow $\mu_{L^u}(\left[y,z_2\right]\setminus \left[y,z_1\right])=\mu_{L^u}((z_1,z_2])=0$ and hence, $z_1=z_2$. Further, every unbounded arc $I$ has the measure $\mu_{L^u}(I)=\infty$ as we showed in Corollary \ref{corol_measure_infinity}. Therefore the map $\pi_p: L^u(p) \times L^s(p) \rightarrow M$ is well-defined by $\left(x,y\right) \mapsto \pi_p(x,y)=z$ such that $\mu_{L^u}\left(\left[p,x\right]\right)=\mu_{L^u}\left(\left[y,\pi_p(x,y)\right]\right)$. \\
\emph{Surjectivity:} let $z \in M$. As the stable and unstable manifold through $z$ is dense we find a point $y \in L^s_p \cap L^u_z$. We have also $L^s_z \cap L^u_p \neq \emptyset$. By moving the interval $[y,z] \subset L^u_y$ via holonomy along the stable foliation to $L^u_p$, we can define $x \in L^u_p$ such that $\mu_{L^u}([p,x]) = \mu_{L^u}([y,z])$. \\
\emph{Continuity:} the point $\pi_p\left((x,y)\right)$ is contained in a product neighborhood and as $\mu_{L^u}$ is holonomy invariant we can conclude that $\pi_p\left(x,y\right) \in L^s(x) \cap L^u(y)$ and $\pi_p$ is continuous (as the canonical coordinates are continuous). 
\emph{Covering:} let $z \in M$ and $N \subset M$ be a product neighborhood for $z$, $N = B^u_{\epsilon}(z) \times B^s_{\epsilon}(z)$ then for $\epsilon > 0$ sufficiently small there are several disjoint stable sets $I^s \subset L^s_p$ such that $I_s \subset L^s_p \cap \mathcal{F}^u(B^s_{\epsilon}(z))$ which are all homeomorphic to $B^s_{\epsilon}(z)$ by holonomy. These correspond to disjoint unstable segments $I^u \subset L^u_p$. So the set $\pi_p^{-1}(N)$ is a union of disjoint sets homeomorphic to $N$.\\
Further, $L^u_p \times L^s_p$ is homeomorphic to $\mathbb{R} \times \mathbb{R}^{q-1}$ as they are stable and unstable leaves, and so the cover is simply connected. It can be concluded that the map $\pi_p: L^u_p \times L^s_p \rightarrow M/\mathcal{F}^c$ is a universal cover for $M/\mathcal{F}^c$. 
\end{proof}
The remaining three steps to conclude the proof follow exactly Hiraide's proof in \cite{H01} and are restated with details and the notation as above in \cite{B12}. To keep this article self-explanatory we sketch the following steps of the proof: we have established $L^u_p \times L^s_p$ as a universal cover of $M/\mathcal{F}^c$ which is homemorphic to $\mathbb{R}^q$. We now show that the fundamental group $\pi_1(M/\mathcal{F}^c)$ is homeomorphic to $\mathbb{Z}^m$ for some $m \in \mathbb{N}$ which is a priori not $q$. 
\begin{lemma}\label{l.fundamentalgroup}
The fundamental group $\pi_1(M/\mathcal{F}^c)$ is homeomorphic to $\mathbb{Z}^m$ for some $m \in \mathbb{N}$. 
\end{lemma}  
\begin{proof}[Proof of Lemma~\ref{l.fundamentalgroup}]
The fundamental group $\pi_1(M/\mathcal{F}^c)$ is isomorphic to the group of all covering transformations $\alpha: L^u_p \times L^s_p \rightarrow L^u_p \times L^s_p$, i.e. $\pi_p \circ \alpha = \pi_p$. We identify both groups in this proof. For $x \in L^u_p$ we denote by $\tilde{\alpha}(x) \in L^u_p$ the first coordinate of $\alpha((x,y))$ for any $y \in L^s_p$. It follows that $\mu_{L^u}(I) = \mu_{L^u}(\tilde{\alpha}(I))$ for any arc $I \subset L^u_p$ as $\mu_{L^u}$ is holonomy invariant. So, $\tilde{\alpha}: L^u_p \rightarrow L^u_p$ is a translation of $\mathbb{R}$ and $\left\{\tilde{\alpha}\;\big|\; \alpha \in \pi_1(M/\mathcal{F}^c)\right\}$ is a free abelian group. Now take $\tilde{\alpha}$ as identity and $q \in L^u_p$. Then $\pi_p: \left\{q\right\} \times L^s_p \rightarrow L^s_p$ is a covering map and as $L^s_p$ is simply connected, $\alpha$ is the identity. Accordingly, we can conclude that $\pi_1(M/\mathcal{F}^c)$ is isomorphic to $\left\{\tilde{\alpha}\;\big|\; \alpha \in \pi_1(M/\mathcal{F}^c)\right\}$ which is isomorphic to $\mathbb{Z}^m$ for some $m \in \mathbb{N}$. 
\end{proof}
The following fact is quite immediatly implied by the minimality of the stable and unstable foliation:
\begin{lemma}\label{l.denseorbit}
The set $\bigcup_{\alpha \in \pi_1(M/\mathcal{F}^c)}\tilde{\alpha}(p)$ is dense in $L^u_p$. 
\end{lemma}
Now we define a map $\tilde{F}:= F|_{L^u_p} \times F|_{L^s_p}$ induced by $F$ on the universal cover and show that $\tilde{F}$ is conjugate to a hyperbolic linear map $A \in \Gl(m,\mathbb{Z})$. The induced map $\tilde{F}_\#$ on the fundamental group is linear. This fact together with Lemma~\ref{l.denseorbit} implies that the map $\tilde{F}$ is a linear expanding map on the first coordinate, so there exists $0 < \lambda < 1$ such that $\mu_{L^u}(\tilde{F}(I))=\lambda^{-1}\mu_{L^u}(I)$. The fundamental group $\pi_1(M/\mathcal{F}^c)$ is isomorphic by an isomorphism $\phi$ to $\mathbb{Z}^m$. So we can define a map $A: \mathbb{Z}^m \rightarrow \mathbb{Z}^m$ by $A:=\phi^{-1} \circ \tilde{F}_\# \circ \phi$. \\
The next stept is the construction of a topological conjugacy $H$ between $\tilde{F}$ and $A$. This can be done as following: recall that $F$ has the unique orbit tracing property. By pulling back the metric from $M/\mathcal{F}$ to the universal cover one can show that the map $\tilde{F}$ has also the unique pseudo orbit tracing property with respect to this pullback metric. Then we construct a bijection $\Phi$ between the lattice $\mathbb{Z}^m$ of $\mathbb{R}^m$ to the lattice $\pi_p^{-1}(p)$ of $L^u_p \times L^s_p$ by taking $\Phi(\alpha(0))=\phi(\alpha)(p,p)$. Now we consider an orbit $\left\{A^iz\right\}$ for some $z \in \mathbb{R}^m$. Then there exists a constant $c > 0$ and a sequence of points $z_i \in \mathbb{Z}^m$ such that $z_i$ is at a distance less than $c$ from $A^iz$. The sequence $\Phi(z_i)$ is then a $\eta$-pseudo orbit of $\tilde{F}$ for some $\eta > 0$, and we can find consequently $\delta >0$ and a point $(x,y)\in L^u_p \times L^s_p$ such that the orbit $\tilde{F}^i(x,y)$ $\delta$-shadows the pseudo orbit $\Phi(z_i)$. So we define $H: \mathbb{R}^m \rightarrow L^u_p \times L^s_p$ by $H(z)=(x,y)$. The map $H$ is a semi conjugacy between $A$ and $F$. It remains to prove the continuity of $H$ and that there is a continuous inverse.   This is done by mapping a point $z \in \mathbb{R}^m$ onto the unique point $(x,y) \in L^u_p \times L^s_p$ which is the shadowing orbit of $\left\{A^ix\right\}$.\\
Up to now, the map $A$ is known to be linear, so we have to show that $A$ is in fact hyperbolic. 
\begin{lemma}\label{l.hyperbolic}
The linear map $A$ as defined above is hyperbolic.
\end{lemma}
\begin{proof}[Proof of Lemma~\ref{l.hyperbolic}]
Suppose that $A$ is not hyperbolic. Then there exists a point $z \in \mathbb{R}^m\setminus\left\{0\right\}$ such that $\left\|Az\right\|= \left\|z\right\|$. Hence, the whole orbit $\left\{A^iz\right\}$ is bounded. This implies that the orbit $\left\{\tilde{F}^iH(z)\right\}$ is also bounded. But then it is $H(z)=(p,p)$ and therefore $z=0$ contradicting the assumption. 
\end{proof}
Thanks to the hyperbolicity of $A$ we can construct an inverse of $H$ by repeating the same procedure above for the map $A$ instead of $\tilde{F}$. We obtain a continuous map $H'$ and show that it is the inverse of $H$. Consequently, the dimensions of $\mathbb{R}^m$ and $L^u_p \times L^s_p$ must be equal, so it is $m=q$, and the conjugacy $H$ on the universal cover decends directly to a topological conjugacy $h$ from $\mathbb{T}^q$ to  $M/\mathcal{F}^c$ between $F$ and the  hyperbolic toral automorphism induced by $A$. This finishes the proof. 

\bibliographystyle{alpha}
\bibliography{bibfile}

\begin{thebibliography}{HRHU10}

\bibitem[BB12]{BB12}
D.~Bohnet and C.~Bonatti.
\newblock Partially hyperbolic diffeomorphisms with uniformly compact center
  foliation: the quotient dynamics.
\newblock arXiv:1210.2835, 2012.

\bibitem[Bin52]{B52}
R.H. Bing.
\newblock A homeomorphism between the 3-sphere and the sum of two solid horned
  spheres.
\newblock {\em Ann. of Math.}, 56(2):354--362, 1952.

\bibitem[Boc45]{B45}
S.~Bochner.
\newblock Compact groups of differentiable transformations.
\newblock {\em Ann. of Math. (2)}, 46:372--381, 1945.

\bibitem[Boh11]{B12}
Doris Bohnet.
\newblock {\em Partially hyperbolic diffeomorphisms with a compact center
  foliation with finite holonomy}.
\newblock PhD thesis, University of Hamburg, 2011.

\bibitem[BP74]{BP74}
M.~I. Brin and Ja.~B. Pesin.
\newblock Partially hyperbolic dynamical systems.
\newblock {\em Izv. Akad. Nauk SSSR Ser. Mat.}, 38:170--212, 1974.

\bibitem[Bro19]{B19}
L.~E.~J. Brouwer.
\newblock {\"{U}}ber die periodischen {T}ransformationen der {K}ugel.
\newblock {\em Mathematische Annalen}, 80:39--41, 1919.

\bibitem[BW05]{BoW05}
C.~Bonatti and A.~Wilkinson.
\newblock Transitive partially hyperbolic diffeomorphisms on 3-manifolds.
\newblock {\em Topology}, 44(3):475--508, 2005.

\bibitem[Car10]{C10}
P.~Carrasco.
\newblock {\em Compact dynamical foliations}.
\newblock PhD thesis, University of Toronto, 2010.

\bibitem[CC00]{CC00}
A.~Candel and L.~Conlon.
\newblock {\em Foliations. {I}}, volume~23 of {\em Graduate Studies in
  Mathematics}.
\newblock American Mathematical Society, Providence, RI, 2000.

\bibitem[CH03]{CH03}
Hellen Colman and Steven Hurder.
\newblock Ls-category of compact hausdorff foliations.
\newblock {\em Trans. Amer. Math. Soc}, 356:1463--1487, 2003.

\bibitem[CK94]{CK94}
A.~Constantin and B.~Kolev.
\newblock The theorem of {K}er\'{e}kj\'{a}rt\'{o} on periodic homeomorphisms of
  the disc and the sphere.
\newblock {\em Enseign. Math.}, 40(3-4):193--204, 1994.

\bibitem[Con78]{C78}
C.~Conley.
\newblock {\em Isolated invariant sets and the {M}orse index}, volume~38 of
  {\em CBMS Regional Conference Series in Mathematics}.
\newblock American Mathematical Society, Providence, R.I., 1978.

\bibitem[Eil34]{E34}
S.~Eilenberg.
\newblock Sur les transformations p{\'e}riodiques de la surface de sph{\`e}re.
\newblock {\em Fund. Math.}, 22:28--41, 1934.

\bibitem[EMT77]{EMT77}
D.B.A. Epstein, K.~Millet, and D.~Tischler.
\newblock Leaves without holonomy.
\newblock {\em Journal of the London Mathematical Society}, s2-16(3):548--552,
  1977.

\bibitem[Eps72]{E72}
D.B.A. Epstein.
\newblock Periodic flows on three-manifolds.
\newblock {\em Ann. of Math.}, 95:66--82, 1972.

\bibitem[Eps76]{E76}
D.B.A. Epstein.
\newblock Foliations with all leaves compact.
\newblock {\em Ann. Inst. Fourier (Grenoble)}, 26(1):265--282, 1976.

\bibitem[EV78]{EV78}
D.B.A. Epstein and E.~Vogt.
\newblock A counterexample to the periodic orbit conjecture in codimension 3.
\newblock {\em Ann. of Math.}, 108(3):539--552, 1978.

\bibitem[Fra70]{F70}
J.~Franks.
\newblock Anosov diffeomorphisms.
\newblock In {\em Proceedings of Symposia in Pure Mathematics}, volume~14,
  pages 61--94. AMS, 1970.

\bibitem[Gog11]{G11}
Andrey Gogolev.
\newblock Partially hyperbolic diffeomorphisms with compact center foliations.
\newblock {\em J. Mod. Dyn.}, 5(4):747--769, 2011.

\bibitem[Hec77]{H77}
G.~Hector.
\newblock Feuilletages en cylindres.
\newblock In {\em Geometry and topology (Proc. {III} Latin Amer. School of
  Math., Inst. Mat. Pura Aplicada CNPq, Rio de Janeiro, 1976)}, pages 252--270.
  Lecture Notes in Math., Vol. 597. Springer, Berlin, 1977.

\bibitem[Hir01]{H01}
K.~Hiraide.
\newblock A simple proof of the franks-newhouse theorem on codimension-one
  anosov diffeomorphisms.
\newblock {\em Ergodic Theory Dynam. Systems}, 21:801--806, 2001.

\bibitem[HK95]{HK95}
B.~Hasselblatt and A.~Katok.
\newblock {\em Introduction to the modern theory of dynamical systems}.
\newblock Cambridge University Press, New York, 1995.

\bibitem[HPS70]{HPS70}
M.~W. Hirsch, C.~C. Pugh, and M.~Shub.
\newblock Invariant manifolds.
\newblock {\em Bull. Amer. Math. Soc.}, 76:1015--1019, 1970.

\bibitem[HRHU10]{HHU10a}
F.~Rodriguez Hertz, M.~A. Rodriguez-Hertz, and R.~Ures.
\newblock A non-dynamically coherent example in 3-torus.
\newblock Preprint, 2010.

\bibitem[Lew89]{L89}
J.~Lewowicz.
\newblock Expansive homeomorphisms of surfaces.
\newblock {\em Bol. Soc. Brasil. Math. (N.S.)}, 20(1):113--133, 1989.

\bibitem[Mar70]{M70}
G.A. Margulis.
\newblock Certain measures associates with u-flows on compact manifolds.
\newblock {\em Functional Anal. Appl.}, 4:55--67, 1970.

\bibitem[Mil75]{M75}
K.~C. Millett.
\newblock Compact foliations.
\newblock In {\em Differential topology and geometry ({P}roc. {C}olloq.,
  {D}ijon, 1974)}, pages 277--287. Lecture Notes in Math., Vol. 484. Springer,
  Berlin, 1975.

\bibitem[New70]{N70}
S.~E. Newhouse.
\newblock On codimension one {A}nosov diffeomorphisms.
\newblock {\em Amer. J. Math.}, 92:761--770, 1970.

\bibitem[OU41]{OU41}
J.~C. Oxtoby and S.~M. Ulam.
\newblock Measure-preserving homeomorphisms and metrical transitivity.
\newblock {\em Ann. of Math. (2)}, 42:874--920, 1941.

\bibitem[PS72]{PS72}
C.~Pugh and M.~Shub.
\newblock Ergodicity of {A}nosov actions.
\newblock {\em Invent. Math.}, 15:1--23, 1972.

\bibitem[Ree52]{R52}
G.~Reeb.
\newblock {\em Sur certaines propri\'et\'es topologiques des vari\'et\'es
  feuillet\'ees}.
\newblock Actualit\'es Sci. Ind., no. 1183. Hermann \& Cie., Paris, 1952.
\newblock Publ. Inst. Math. Univ. Strasbourg 11, pp. 5--89, 155--156.

\bibitem[Rud87]{R87}
Walter Rudin.
\newblock Real and complex analysis, 1987.

\bibitem[Sul76a]{S76a}
D.~Sullivan.
\newblock A counterexample to the periodic orbit conjecture.
\newblock {\em Inst. Hautes \'Etudes Sci. Publ. Math.}, (46):5--14, 1976.

\bibitem[Sul76b]{S76b}
D.~Sullivan.
\newblock A new flow.
\newblock {\em Bull. Amer. Math. Sco.}, 82(2):331--332, 1976.

\bibitem[Vie99]{V00}
J.~Vieitez.
\newblock A 3{D}-manifold with a uniform local product structure is {T}3.
\newblock {\em Publ. Mat. Urug.}, 8:47--62, 1999.

\bibitem[vK19]{K19}
B.~von Ker{\'e}kj{\'a}rt{\'o}.
\newblock {\"{U}}ber {T}ransformationen des ebenen {K}reisringes.
\newblock {\em Mathematische Annalen}, 80:33--35, 1919.

\bibitem[Wil98]{W98}
A.~Wilkinson.
\newblock Stable ergodicity of the time-one map of a geodesic flow.
\newblock {\em Ergodic Theory Dynam. Systems}, 18(6):1545--1587, 1998.

\end{thebibliography}
\end{document}